\documentclass[11pt,letterpaper]{article}
\usepackage{amsfonts}
\usepackage{amsmath,amssymb,color}
\usepackage{graphicx}
\usepackage{enumitem}
\usepackage{multirow}
\usepackage{bm}
\usepackage[doublespacing]{setspace}
\newcommand{\jc}[1]{{{\small \color{red}\sc  (#1)}}}
\newcommand{\aal}{\boldsymbol \alpha}

\newcommand{\mA}{\mbox{$\mathcal A$}}

\newcommand{\pp}{\mbox{$\mathbf p$}}

\newcommand{\qq}{\mbox{$\mathbf q$}}

\def\uu{\mbox{$\mathbf u$}}

\newcommand{\YY}{\mbox{$\mathbf Y$}}
\newcommand{\yy}{\mbox{$\mathbf y$}}

\newtheorem{theorem}{Theorem}

\newtheorem{corollary}{Corollary}

\newtheorem{definition}{Definition}
\newtheorem{example}{Example}

\newtheorem{lemma}{Lemma}

\newenvironment{proof}[1][Proof]{\noindent\textbf{#1.} }{\ \rule{0.5em}{0.5em}}

\title{On the Identifiability of Diagnostic Classification Models}
\author{Guanhua Fang, Jingchen Liu, and Zhiliang Ying
}

\begin{document}
\maketitle

\baselineskip=25pt

\begin{abstract}
	This paper establishes fundamental results for statistical inference of diagnostic classification models (DCM).
	The results are developed at a high level of generality, applicable to essentially all diagnostic classification models.
	In particular, we establish identifiability results of various modeling parameters, notably item response probabilities, attribute distribution, and $Q$-matrix-induced partial information structure.
    Consistent estimators are constructed.
    Simulation results show that these estimators perform well under various modeling settings.
    We also use a real example to illustrate the new method.
	The results are stated under the setting of general latent class models. For DCM with a specific parameterization,  the conditions may be adapted accordingly.
\end{abstract}

\section{Introduction}

Cognitive diagnosis is an important area which has become increasingly popular in educational assessment, psychiatric
evaluation, and many other disciplines.
An incomplete list of recent developments of cognitive diagnosis models include the rule space method \cite{Tatsuoka1985, Tatsuoka}, the reparameterized unified/fusion model (RUM) \cite{DiBello, Hartz, He}, the conjunctive (noncompensatory) DINA and NIDA models \cite{Junker, TatsuokaC, dela, Templin2006}, the compensatory DINO and NIDO models \cite{Templin, Templin2006}, the attribute hierarchy method \cite{AHM},  clustering methods \cite{Chiu}, and the G-DINA model \cite{delaTorre}; see also \cite{Junker0, von, Rupp} for more approaches to cognitive diagnosis.
	A comprehensive review of diagnostic classification models can be found in \cite{Rupp2008}. Statistical inference methods have also been developed for these models, such as item parameter estimation/calibration  \cite{RoussosTH, Stout2007},
	$Q$-matrix estimation \cite{Chen, CDCAT, DATAdriven, Selflearn}.

	The primary focus of this paper is to develop a theoretical foundation for statistical inferences of diagnostic classification models.
	More precisely, we present identifiability results for various model parameters. We aim at developing results under a general framework that is applicable to essentially all diagnostic classification models.
	To do so, we cast the diagnostic classification models under the framework of a more general family, i.e. the latent class models, and consider DCMs as latent class models with special parameterizations or constraints. 
	
	Parameter identifiability is an important and long-standing issue in latent variable/factor models.
	This is mostly due to the fact that latent variables are not directly observed and the observed data provide limited information of the model parameters. Nonidentifiability is often present in such models when proper constrains are not imposed.
	In exploratory factor analysis, such as multidimensional item response theory models, the factor loading matrix is  identifiable  only up to a rotation and a reflection, which is mathematically equivalent to performing an orthogonal transformation of the continuous latent factors (J\"{o}reskog \cite{Joreskog}).
	In the context of cognitive diagnosis, the latent variables are discrete and cannot be rotated or reflected. Nonetheless, identifiability issues remain an open problem for general DCMs.
	In this paper, we develop sufficient conditions, under which the model parameters can be consistently estimated based on the data, in most cases, by the maximum likelihood estimator.
	
	The technical discussion falls into two parts. We first consider the parameters for the item response functions and the attribute population.
	These parameters live on a continuous space and are considered as ``regular'' parameters compared to the item partial information that is described in the subsequent discussion.
	In the literature of latent class models, there are existing identifiability results in this regard. For instance, recent work in \cite{Allman} discusses \emph{generic identifiability} of general latent class models, hidden Markov models, and several other models with latent variables.
	The identifiability results are \emph{generic} in the sense that they hold when the parameters do not lie on some measure zero set. This set is never clearly specified either in the statements or in the technical proofs.
	As we will discuss in the sequel, the parameters of diagnostic classification models always live on a low-dimensional manifold that is indeed a measure zero set. Therefore, the results in \cite{Allman} cannot be applied in this context.
	The results of this paper stand on the fundamental results of Kruskal \cite{Kruskal} concerning the rank of three-dimensional arrays.
	We investigate the special structures of diagnostic classification models, adapt Kruskal's analysis to this particular context, and develop identifiability results.

	In the analysis of diagnostic classification models, Bayesian models and inference are popular. Technically speaking, model identifiability  is a lesser issue for Bayesian inference that is valid as long as the posterior distribution is well defined.
	On the other hand, this seemingly plausible solution in fact could be misleading in that, in case of nonidentifiability, the likelihood function does not provide information to differentiate among certain parameters and thus it all relies on the prior information that is mostly subjective.
    Computational issues may also exist under this situation.
	Thus, identifiability results are of importance to Bayesian models as well.

	Besides developing identifiability results for the item parameters, we also provide discussion of the $Q$-matrix that is a key quantity in the specification of diagnostic classification models.
	The $Q$-matrix provides a qualitative description of the item-attribute relationship. However, this precise relationship depends on the specific model parameterization and the interaction among attributes.
	This paper develops general results applicable to essentially all DCMs instead of specific parametric models. We consider a slightly different estimand implied by the $Q$-matrix, the partial information structure for each item that will be defined in the subsequent discussion.
	We develop identifiability results for the partial information structure for each item. Under specific model parameterizations, the $Q$-matrix can also be reconstructed based on the estimated partial information for each item. This will be illustrated via a real data analysis.

	To illustrate the identifiability results, we consider estimation of the item response parameters, the attribute distribution, and the partial information structure under a latent class model with Dirichlet allocation.
	This model is proposed for exploratory analysis. It contains infinitely many mixtures and thus is a saturated and over-parameterized model.
	To regularize the overfitting (the number of latent classes), the model adopts a prior distribution for the attribute population via the stick-breaking representation that is originated from the derivation of the Dirichlet processes.
	In the analysis, we adopt a full Bayesian setting. The posterior distribution can be obtained by means of a Gibbs sampling scheme without truncation to a finite mixture model.
	An estimator of the partial information structure is proposed based on a clustering algorithm combined with a Bayesian estimator.
	We illustrate the identifiability results by fitting this model to simulated data generated from several diagnostic classification models and a real data set. For the latter, we reconstructed a  parametric loading structures  and its $Q$-matrix.

	The rest is organized as follows. In Section \ref{SecPre}, we present basic concepts  including latent class models, diagnostic classification models, identifiability, and estimation consistency.
	In Section \ref{SecId}, the main results are presented including the identifiability of item parameters, attribute distribution, and the partial information structure.
	The latent class model with  Dirichlet allocation along with its inference is presented  in Section \ref{SecModel}.
	Simulation studies and real data analysis are presented in Sections \ref{SecSim} and \ref{SecReal}, respectively.

\section{Preliminaries}\label{SecPre}

\subsection{Latent class models and diagnostic classification models}\label{SecLat}

We start with descriptions of the latent class models and the diagnostic classification models and establish their connections.
Consider a $J$-dimensional multivariate categorical response random vector $\YY= (Y^1,...,Y^J)$. We use subscript to index independent replications, that is, $\YY_i$ and $Y_i^j$. The random vector ($\YY$) or variable ($Y^j$) without a subscript denotes a generic random vector or variable.
Let $Y^j$ be a discrete random variable taking $k_j$ possible values. Without loss of generality, let $Y^j\in \{1,..., k_j\}$.
In the formulation of latent class models, the dependence among $Y^1$,..., $Y^J$ is induced by a discrete latent variable (latent class) $\aal$ taking values in a discrete set $\mA$. We similarly use $\aal_i$ to denote its independent replications.
Conditional on $\aal$, the distribution of $\YY$ falls into some simple form. In this paper, we assume that  $Y^1$, ..., $Y^J$ are conditionally independent given $\aal$, that is,
\begin{equation*}
P(Y^j = y^j, j=1,...,J|\aal) = \prod_{j=1}^J P(Y^j = y^j |\aal).
\end{equation*}
This is known as the \emph{local independence} assumption.
The $\pi_{\aal}$ is the probability mass function of the latent class membership $\aal$.
Suppose conditional distribution is parameterized through
$$P(Y^j = y^j |\aal) = f_{j}(y^j |\theta,\aal),$$
where $\theta_j$ is the item parameter specific to item $j$. Thus, the joint marginal distribution of $\YY$ can be expressed as
\begin{equation}\label{LCM}
P(Y^1 = y^1, ..., Y^J, = y^J) = \sum_{\aal \in \mA}\Big \{ \prod_{j=1}^J P(Y^j = y^j |\aal) \pi_{\aal}\Big \}.
\end{equation}

Diagnostic classification models in general admit the same distribution as in \eqref{LCM}. However, it often imposes additional parametric or low-dimensional structures so that the item response function $f_j$ and the latent variable $\aal$ have practical interpretations.
Here, we adopt the following parameterization.
The latent variable $\aal$ is parameterized by a multidimensional binary or multi-category vector, i.e. $\aal \in  (\alpha^1,...,\alpha^K)$
where  $\alpha^k$ takes $d_k$ possible values. Each $\alpha^k$ is known as an attribute or trait indicating the presence/absence or level of a latent characteristic.

	In addition to the multidimensional structure of the latent class space, the item response function $f_j$ also admits some parametric and usually low-dimensional structures that distinguish themselves from the general latent class models.
	To elaborate on this feature, we need to introduce the $Q$-matrix that is a $J\times K$ matrix taking binary values and indicating the item-attribute association.
Each row of $Q$ corresponds to an item and each column corresponds to an attribute. We write
\begin{equation*}
Q = \left (
\begin{array}{c}
q_1\\
\vdots\\
q_J
\end{array}
\right),
\end{equation*}
where $q_j =(q_{j1},...,q_{jK})$ is a $K$-dimensional row vector. As mentioned previously, $q_{jk} = 1 \textrm{ or } 0$, indicates whether or not $f_j$ depends on $\alpha^k$.
	For instance, if $q_j = (1,1,0)$, then
$$f_j(y|\theta, \aal) = f_j(y | \theta, \alpha^1, \alpha^2).$$
The $Q$-matrix qualitatively describes  the relationship between items and attributes. The specific forms of  $f_j$ as functions of $\aal$ are determined by the parametric model subject to the constraints implied by the $Q$-matrix.
The specification also consists of additional item parameters.
An important feature of the item response function is that $f_j(y| \theta, \aal)$ as a function of $\aal$ usually does not necessarily depend on all $\alpha^k$, $k = 1, \ldots, K$.

\subsection{Examples of diagnostic classification models}

In what follows, we provide a few examples of parametric diagnostic classification models.
\begin{example}[DINA model, \cite{Junker}]\label{DINA}
For item $j$ and attribute vector $\aal$, we define the ideal response
\begin{equation}\label{xidina}
\xi_{DINA}^j(\aal,Q)= \prod_{k=1}^K{(\alpha_{k})}^{q_{jk}} = I (\alpha_k \geq q_{jk}\mbox{ for all }k)
\end{equation}
that is, whether $\aal$ has all the attributes required by item $j$.
For each item, there are two additional parameters $s_j$ and $g_j$, known as the slipping and  guessing parameters. The response probability $p_{j,\aal}$ takes the form
\begin{equation}\label{DDINA}
p_{j,\aal} = (1-s_j)^{\xi_{DINA}^j(\aal,Q)}g_j^{1-\xi_{DINA}^j(\aal,Q)}.
\end{equation}
If $\xi_{DINA}^j(\aal,Q)=1$ (the subject is capable of solving a problem), then the positive response probability is $1-s_j$; otherwise, the probability is $g_j$.
The item parameter vector is $ (s_j,g_j:~ j=1,\cdots,J )$.

The DINA model assumes a conjunctive (non-compensatory) relationship among attributes.
It is necessary to possess all the attributes indicated by the $Q$-matrix to be capable of providing  a positive response. In addition, having additional unnecessary attributes does not compensate for the lack of necessary attributes.
The DINA model is popular in the educational testing applications and is often used for modeling exam problem solving processes.
\end{example}

\begin{example}[NIDA model]\label{NIDA}
The NIDA model admits the following form
$$p_{j,\aal} = \prod_{k=1}^K [(1-s_k)^{\alpha_k} g_k ^{1-\alpha_k}]^{q_{jk}}.$$
The problem solving involves multiple skills indicated by the $Q$-matrix. For each skill, the student has a certain probability of implementing it: $1-s_j$ for mastery and $g_j$ for non-mastery. The problem is solved correctly if all required skills have been implemented correctly by the student, which leads to the above positive response probability.
\end{example}

The following reduced NC-RUM model is also a conjunctive model, and it generalizes the DINA and the NIDA models by allowing the item parameters to vary among attributes.

\begin{example}[Reduced NC-RUM model]\label{RUM}
Under the reduced noncompensatory reparameterized unified model (NC-RUM), we have \begin{equation}
p_{j,\aal} = \phi_j \prod_{k=1}^K(r_{jk})^{q_{jk}(1-\alpha_k)},
\end{equation}
where $\phi_j$ is the correct response probability for subjects who possess all required attributes and $r_{j,k}$, $0<r_{j,k}<1$, is the penalty parameter for not possessing the $k$th attribute. The corresponding item parameters are
 $(\phi_j, r_{j,k}: j=1,\cdots,J, k=1,\cdots,K).$
\end{example}

In contrast to the  DINA, NIDA, and Reduced NC-RUM models, the following DINO and C-RUM models assume compensatory (non-conjunctive) relationship among attributes, that is, one only needs to possess one of the required attributes  to be capable of providing a positive response.

\begin{example}[DINO model]\label{DINO}
The ideal response of the DINO model is given  by
\begin{equation}\label{xidino}
\xi_{DINO}^j(\aal,Q)= 1- \prod_{k=1}^K(1-\alpha_{k})^{q_{jk}} =  I(\alpha_k \geq q_{jk}\mbox{ for at least one }k).
\end{equation}
Similar to the DINA model, the positive response probability is
$$p_{j,\aal} = (1-s_j)^{\xi_{DINO}^j(\aal,Q)}g_j^{1-\xi_{DINO}^j(\aal,Q)}.$$
The DINO model is the dual model of the DINA model.
	The DINO model is often used in the application of psychiatric assessment, for which the positive response to a diagnostic question (item) could be due to the presence of one disorder (attributes) among several.
\end{example}

%

\begin{example}[C-RUM model]\label{CRUM}
The GLM-type parametrization with a logistic link function is used for the compensatory reparameterized unified model (C-RUM), that is
\begin{equation}\label{crum}
p_{j,\aal} = \frac{\exp(\beta^j_{0}+\sum_{k=1}^K\beta^j_{k}q_{jk}\alpha_{k})}{1+\exp(\beta^j_{0}
+ \sum_{k=1}^K\beta^j_{k}q_{jk}\alpha_{k})}.
\end{equation}
The corresponding item parameter vector is  $(\beta^j_{k}: j=1,\cdots,J, k=0,\cdots,K).$
The C-RUM model is a compensatory model and one can recognize (\ref{crum}) as a structure in multidimensional IRT model or in factor analysis.
\end{example}

\subsection{Identifiability of model parameters}

We discuss the identifiability for two types of parameters separately: 1. item parameters and the attribute distribution, 2. item partial information and the $Q$-matrix. Let $\theta$ denote the vector of all item parameters and the attribute distribution. Its identifiability is defined as follows.
\begin{definition}\label{DefId}
The parameter $\theta$ is said to be identifiable if, for any $\theta' \neq \theta$, the resulting marginal distributions of the responses $\YY$ in \eqref{LCM} are distinct.
\end{definition}
If the parameter $\theta$ is identifiable, then, thanks to the entropy inequality \cite{Kullback} and under very mild conditions, the maximum likelihood estimator is consistent.

The identifiability of the $Q$-matrix is different from that of the regular item parameters.
In what follows, we provide some discussions on the estimation of the $Q$-matrix for a generic diagnostic classification model. One may view this problem from different perspectives. The most straightforward approach is to treat $Q$ as part of the model parameters and to consider it as a usual estimation problem. This is often difficult from the computational aspect in that $Q$ is a discrete matrix living on a high dimensional space, in particular, $Q\in \{0,1\}^{J\times K}$. Even with a reasonably small number of items and a few attributes, this space is often too large to explore thoroughly by any existing numerical method as the dimension grows exponentially fast with both $J$ and $K$.
Estimators developed based on this idea, even though theoretically sound, often suffer from substantial computational overhead.
	One of such instances is the maximum likelihood estimator of $Q$. Generally speaking, optimizing a discrete and nonlinear function over $\{0,1\}^{J\times K}$ is  computationally  intensive and sometimes infeasible. This approach does not take advantage of the special structures of the $Q$-matrix and further of the likelihood function.

A different approach is to cast the $Q$-matrix estimation in the context of variable selection. Consider the item response function $f_j(y |\theta, \aal)$. If both the response $Y^j$ and the latent variable $\aal$ were observed, then the estimation of $Q$ is a regular variable selection problem.
In most situations, $f_j$ takes the form of a generalized linear model, in which the responses to items are the dependent variables, the attributes play the role of covariates, and the item parameters $\theta$ are the regression coefficients.
Thus, the $Q$-matrix estimation is equivalent to a variable selection problem. However, in the context of latent class models, the covariates $\aal$'s are all missing and therefore the task is, rigorously speaking, to select latent variables. Chen et al. \cite{Chen} took this viewpoint and developed estimation methods for the $Q$-matrix via regularized likelihood.

The last approach is similar to the previous one, but is more generic and is the primary focus of the current analysis.
 The introduction of the $Q$-matrix suggests that a single item usually does not provide information to differentiate all dimensions of the attribute profile.
	In particular,  $q_{jk} = 0$ means that item $j$ is irrelevant to attribute $k$.
Under the setting of  latent class models (not necessarily possessing a specific parameterization), this corresponds to an item-specific \emph{partial information structure}. Each particular attribute profile $\aal$ in the DCM parameterization corresponds to one latent class. If an item does not differentiate all dimensions of $\aal$, then some distinct attribute profiles may admit the same response distribution. In other words, there exist $\aal_1 \neq \aal_2$ such that $f_j(y|\theta, \aal_1) = f_j(y|\theta, \aal_2)$  for all $y$.
In this case, responses to this item of subjects in  latent classes $\aal_1$ and $\aal_2$ admit the same probability law.
Thus, each item usually provides partial information of the entire attribute profile. This information structure will be formulated mathematically in the sequel.
Each $Q$-matrix along with a specific model parameterization (such as the DINA and DINO models, etc) maps to a unique item-specific partial information structure.

\section{On the identifiability of diagnostic classification models}\label{SecId}

We present the main identifiability results in this section. In the current formulation, the identifiability of diagnostic classification models consists of two components: 1. the item parameters and the attribute population, 2.  the partial information structure of each item.

\subsection{Identifiability of item parameters and the attribute distribution}

We first present four theorems that are applicable to different situations. We start with the simplest case that the responses are binary and each item only has two possible response distributions.
The binary response settings will be relaxed in subsequent theorems.
The result is applicable to stylized models such as the DINA and the DINO models.

\begin{theorem}\label{ThmBinaryBinary}
We consider the general setting of a latent class model with $M>2$ latent classes.
The responses are binary and take values in $\{0,1\}$.
For each item $j$, let $p_{j\aal} = P(Y_j= 1 | \aal) $. Let $\pi_{\aal}$ be the probability of each latent class. Suppose that the following assumptions are satisfied.
\begin{enumerate}
\item[A1] There exist three non-overlap subsets of items denoted by $I_1$, $I_2$, and $I_3$ such that for each $\aal_1 \neq \aal_2$ and $l=1,2$, and $3$,  the conditional distributions of $(Y_j: j \in I_l)$ on classes $\aal_1$ and $\aal_2$ are distinct.
\item[A2] For each $j\in I_1\cup I_2 \cup I_3$, the response probabilities  $(p_{j\aal}: \aal = 1,..., M)$ take only two possible values.
\item[A3] $\pi_{\aal} > 0$ for $\aal = 1,2, \ldots, M$.
\end{enumerate}
Then, the item parameters $p_{j\aal}$ and the latent class population $\pi_{\aal}$ are identifiable up to a permutation of the class label.
\end{theorem}
	For the latent class models, the class labels are not identifiable based on the data. The identifiability up to a permutation of class label refers to the following fact.
	Suppose that there exist a set of item parameters and attribute prior distribution, denoted by $\tilde p^{k}_{j\aal}$ and $\tilde \pi_{\aal}$, yielding the same marginal distribution as that in \eqref{LCM} with parameters $p^{k}_{j\aal}$ and $\pi_{\aal}$.
	Then, there exists a permutation of the class labels $\lambda$ such that $$\tilde p^{k}_{j\lambda(\aal)}= p^{k}_{j\aal}\quad \textrm{and} \quad \tilde \pi_{\lambda(\aal)} =  \pi_{\aal}.$$
Theorem \ref{ThmBinaryBinary} relies on several assumptions.
Assumption $A2$ requires that each item response function take only two values. This is applicable to simple models, such as the DINA and the DINO model.
In the context of educational testing, the population is often split to two groups: capable and incapable. In psychiatric assessments, the partition corresponds to the presence or absence of mental disorders.
 Regarding assumption $A$1 for a test, it is necessary to have at least one set of items $I$ such that the conditional distributions of $(Y_j: j\in I)$ are different for all $\aal$, that is, the items in $I$ differentiate among all latent classes; otherwise, it is always possible to merge some latent classes and reduce the model to satisfy the condition.
Following the idea of repeated measurements \cite{Repeat}, Assumption $A$1 requires three such sets of items. It is often satisfied for tests with reasonably high accuracy.

The following corollary of Theorem \ref{ThmBinaryBinary} presents easy-to-check conditions  for  diagnostic classification models. The corollary does not assume a particular paramerization.


\begin{corollary}\label{ThmB}
Consider a diagnostic classification model for $J$ binary responses with $K$ binary attributes.
Suppose that we can rearrange the columns and rows of $Q$ such that it contains three distinct identity submatrices. Then, the item response function and the attribute population are identifiable up to a relabeling of the latent classes.
\end{corollary}

Model identifiability is closely related to the completeness of $Q$-matrix.
A $Q$-matrix is said to be complete if it can differentiate all latent classes.
It is well known that for a $Q$-matrix to be complete, it must have a $K$ by $K$ identity submatrix. \cite{Chiu}
However, having an identity is not sufficient enough.
In our corollary \ref{ThmB} above, we require three identity matrices structure in $Q$-matrix to ensure a stronger sufficient conditions for identifiability of parameters.

The following theorem generalizes the results to the cases of multi-category responses.

\begin{theorem}\label{ThmBinary}
We consider the general setting of a latent class model with $M>2$ classes. The response to item $j$ takes $k_j$ possible values $\{1,...,k_j\}$. Let $p_{j\aal}^k = P(Y_j= k | \aal) $ and $\pp_{j\aal} = (p_{j\aal}^1,..., p_{j\aal}^{k_j})$ be the response vector.  Let $\pi_{\aal}$ be the probability of each latent class. Suppose that the following conditions are satisfied for each item $j$.
\begin{enumerate}
\item[B1] The response vector  $\pp_{j\aal}$ takes only two distinct values for all $\aal = 1,..., M$.
\item[B2] There exist three non-overlap subsets of items denoted by $I_1$, $I_2$, and $I_3$ such that for $\aal_1 \neq \aal_2$ and $l=1,2$, and $3$,  there exists $j\in I_l$ such that
$$p_{j\aal_1}^1 + ... + p_{j\aal_1}^k \neq p_{j\aal_2}^1 + ... + p_{j\aal_2}^k$$
for some $k=1,...,k_j-1$.
\item[B3] $\pi_{\aal} > 0$ for $\aal = 1,2,\ldots,M$.
\end{enumerate}
Then, the item parameters $p^{k}_{j\aal}$ and the latent class population $\pi_{\aal}$ are identifiable up to a permutation of the class labels.
\end{theorem}


We now proceed to the results of the general diagnostic classification models in which the responses are multi-categorical variables and the response function may take more than two possible values.
We need to introduce the $T$-matrix. Consider a test of $K$ attributes and $J$ items. The response to item $j$ takes $k_j$ different values $\{1,...,k_j\}$ and attribute $i$ takes $d_i$ possible different values $\{1,...,d_i\}$. There are $\kappa = \prod_{j=1}^J k_j$ possible response patterns and $M = \prod_{k=1}^K d_k$  latent classes.
 We defined the $T$-matrix as a $\kappa \times M$ matrix. Each column of the matrix corresponds to one attribute profile or one latent class and each row corresponds to one response pattern. The particular order of the columns and rows of the $T$-matrix does not affect the results. For technical convenience, we use $\aal$ and $\yy = (y^1,...,y^J)$ to label its columns and rows, that is, $t_{\yy\aal}$ is the element in  row $\yy$ and  column $\aal$.
\begin{equation*}
t_{\yy\aal} = P(\YY = \yy |\aal) = \prod_{j=1}^J p_{j\aal}^{y^j}.
\end{equation*}
Very often, we construct a $T$-matrix for a subset of items $I\subset \{1,...,J\}$. We use $T_I$ to denote the corresponding matrix of the items in $I$ and $T$ to denote the matrix to all items.
We first present a theorem that is essentially applicable to all situations, but its conditions are sometimes difficult to check and then we provide a few propositions for specific situations.

\begin{theorem}\label{ThmGeneral}
Consider a diagnostic classification model for $J$ items with $K$ attributes. Suppose that we are able split the items into three nonoverlap subsets $I_1$, $I_2$, and $I_3$, that is, $I_1\cup I_2 \cup I_3 = \{1,...,J\}$. If $T_{I_1}$, $T_{I_2}$, and $T_{I_3}$, the $T$-matrices corresponding to each of the three subsets, are all of full column rank. The probability of each latent attribute profile class $\pi_{\aal}$ is positive. Then the item parameters $p^{k}_{j\aal}$ and the latent class population $\pi_{\aal}$ are identifiable up to a permutation of the class labels.
\end{theorem}

The above theorem provides sufficient conditions for the identifiability.
These conditions are very weak but are sometimes difficult to check in practice. This is mostly due to the fact that the $T$-matrix is computationally costly to construct. For instance, consider a subset of 20 items having binary responses and its $T$-matrix has $2^{20} = 1,048,576$ rows. Thus, construction of $T$-matrix for reasonably large-scale studies is impossible. For this concern, we present the following theorem that provides stronger but much easier to check conditions. It requires the constructions of much smaller $T$-matrices.

\begin{theorem}\label{ThmGeneralDCM}
Consider a diagnostic classification model for $J$ items with $K$ multi-category attributes. For each attribute $k= 1,...,K$, there exist three nonoverlap subsets of items $I_{1,k}$, $I_{2,k}$, and $I_{3,k}$ satisfying the following conditions.
\begin{enumerate}
\item [C1] The items in $I_{i,k}$ are only associated with attribute $k$, that is, their corresponding row vector in $Q$ is $e_k$.
\item [C2] Let $T_{I_{i,k}}$ be the corresponding $T$-matrix of this reduced simple-attribute model. The matrix $T_{I_{i,k}}$ is of full column rank.
\item [C3] Let $\pi_{\aal}$ be the probability of latent class $\aal$. $\pi_{\aal} > 0, \aal = 1,2, \ldots, M$.
\end{enumerate}
Then, the item parameters $p^{k}_{j\aal}$ and the latent class population $\pi_{\aal}$ are identifiable up to a permutation of the class labels.
\end{theorem}

	The above theorem seemingly requires many single-attribute items. In practice, each of the subset $I_{i,k}$ usually contains very few, in fact most of the time, single items.
	However, notice that the matrix $T_{I_{i,k}}$ for the reduced single-attribute model contains $d_k$ columns. In the case of binary attribute, it is sufficient to include a single item in each $I_{i,k}$; see the proof of Theorem \ref{ThmBinary}.
	It may remain possibly sufficient to include a single item in each $I_{i,k}$ if the response to the item also takes more than two possible values.
	Generally speaking, we need to include sufficiently many items in each $I_{i,k}$ so that their responses contain information to differentiate different latent classes defined by attribute $k$.
	Furthermore, the construction of the $T$-matrices for the reduced model is much easier as $I_{i,k}$ often contains very few items and the matrix only contains $d_k$ columns.

\subsection{Identifiability of partial information structure}

	The previous subsection provides results on the identifiability of the item parameters and the attribute population. We now proceed to a discussion of the $Q$-matrix.
	The $Q$-matrix provides a qualitative description between the item-attribute relationship. The specific form an item response function takes depends on the model parameterization and the loading structure.
	As the aim of this study is to provide results applicable to general diagnostic classification models, we take a slightly different viewpoint and state the identifiability of the partial information structure for each item that is mathematically a more general concept than the $Q$-matrix.
	
	To proceed, we start with a description of the partial information of an item in the context of a general latent class model. Let $\aal \in \mathcal M = \{1,...,M\}$ denote the latent class membership.
	The partial information of item $j$ characterizes the latent classes it is capable of differentiating.
	Mathematically, we define an item-specific equivalence relation on $\mathcal M$, denoted by ``$\overset j = $''. For $\aal_1, \aal_2 \in \mathcal M$, $$\aal_1 \overset j = \aal_2 \quad \textrm{if} \quad p_{j\aal_1}^y = p_{j\aal_2}^y\quad \textrm{for all $y = 1,...,k_j$.}$$
	It is not hard to verify that ``$\overset j = $'' is an equivalence relation. We define the quotient set $\mathcal M / \overset j =$ as the partial information of item $j$ and use $[\aal]_j$ to denote the corresponding equivalence class that latent class $\aal$ belongs to.
    The map $[~\cdot~]_j$ is known as the canonical projection which leads to a partition of latent classes.
	Two latent classes are mapped to the same equivalence class, $[\aal_1]_j = [\aal_2]_j$ if $\aal_1 \overset j = \aal_2$ and, in this case, item $j$ does not provide information to differentiate $\aal_1$ and $\aal_2$.

From the modeling point of view, the $Q$-matrix along with a particular loading parameterization determines the partial information of each item.
Consider a particular item $j$ whose corresponding row vector in $Q$ has the first $l$ entries being one and others being zero.
Then, the conditional response distribution reduces to
$$P(Y^j = y^j | \aal) = P(Y^j = y^j | \alpha^1,...,\alpha^l).$$
Consider two attribute profiles $\aal_1$ and $\aal_2$. If their first $l$ components are identical, then $[\aal_1]_j = [\aal_2]_j$. Under some simple loading structures (for instance, the DINO model), even if some of their first $l$ elements are not identical, $\aal_1$ and  $\aal_2$ may still belong to the same equivalence class.
The following theorem presents identifiability of the partial information structure of each item.

\begin{theorem}\label{ThmPI}
Under each set of conditions of Theorems \ref{ThmBinaryBinary}, \ref{ThmBinary}, \ref{ThmGeneral}, or \ref{ThmGeneralDCM}, the partial information of each item can be consistently estimated up to a permutation of the latent class label. That is, letting ``$\langle ~~\rangle_j$'' be the estimated canonical projection of item $j$, there exists a permutation of latent class labels $\lambda$ such that
$$P( \langle \lambda(\aal) \rangle_j = [\aal]_j) \to 1$$
as the sample size $n\to \infty$, $\aal \in \mathcal M$.
\end{theorem}

\section{Estimation via a latent class model with Dirichlet allocation} \label{SecModel}

	In this section, we provide estimation for the partial information structure in the context of a latent class model using Dirichlet allocation \cite{Dunson}.
	We adopt the general setup of the latent class models in Section \ref{SecLat}. In the exploratory analysis, the number of latent classes is usually unknown. In this analysis, we do not assume an upper bound on the number of latent classes. Without loss of generality, we assume that the latent classes are labeled by natural numbers $\mathcal M = \{1,2,...,\}$. The marginal distribution of  the responses in \eqref{LCM} becomes
\begin{equation}\label{like}
P(Y^1 = y^1, ..., Y^J, = y^J) = \sum_{\aal =1}^\infty \pi_{\aal} \prod_{j=1}^J P(Y^j = y^j |\aal).
\end{equation}
	The unknown parameters are $\pp$ and $\pi$. This infinite mixture model is overly parameterized as it is unnecessary to include infinitely many latent classes for which the estimation is impossible.
	We adopt a Bayesian model and a proper prior distribution to regularize this over parameterization. The item response probabilities follow a Dirichlet prior distribution
\begin{equation}\label{priorP}
\pp_{j\aal}= (p_{j\aal}^y: y = 1,..., k_j) \sim Dirichlet (1,...,1).
\end{equation}
The attribute distribution $\pi$ is an infinite-dimensional vector summing up to one. We adopt a stick-breaking \cite{Sethur} representation for its prior.
In particular, let $\{V_i: i= 1,2,...\}$ be a sequence of i.i.d.~random variables following the Beta distribution Beta$(1,\beta)$. For each natural number $\aal$, the prior distribution of $\pi_{\aal}$ has the following representation
\begin{equation}\label{priorPi}
\pi_{\aal} = V_{\aal} \prod _{l< \aal} (1-V_l).
\end{equation}
It is easy to verify that $\pi_{\aal}$ under the above construction is a well defined probability mass function.
This is known as the stick-breaking representation originated from the Dirichlet process.
We borrow this representation mostly due to its technical convenience for modeling a discrete distribution.
The likelihood function \eqref{like} and the prior distributions \eqref{priorP} and \eqref{priorPi} completely specify a Bayesian model.

	We adopt this model for several reasons. First, it does not require to specify the number of latent classes. The stick-breaking representation penalizes the ``tail'' latent classes.
	In addition, the posterior distribution of this model can be obtained through a sliced sampler \cite{Walker} that is a Gibbs sampler via a data augmentation scheme without truncating the model to a finite mixture.
	Given that this is not the emphasis of this paper, we present the posterior simulation scheme in the appendix.
	
	Henceforth, we assume that the posterior distribution of the model parameters has been obtained numerically.
	We consider posterior mean as the point estimator, that is,
	$$(\hat {\pp}, \hat {\pi}) = E\{(\pp,\pi) | \yy_1,...,\yy_n\}.$$
	It is well known  that the posterior mean has the same asymptotic distribution as that of the maximum likelihood estimator; see, for instance, \cite{Postmean}.
	The item-wise partial information is estimated by clustering the item response probabilities. For each item $j$, let $\hat \pp_{j\aal} = (\hat p^1_{j\aal},..., \hat p^{k_j}_{j\aal})$ be its item response distribution of latent class $\aal$.
	We treat $\hat \pp_{j1}$, $\hat \pp_{j2}$, $\hat \pp_{j3}$, ... as (infinitely many) samples, each of which is a $k_j$ dimensional vector, and apply the $K$-means clustering algorithm.
	Although there are seemingly infinitely many samples, $\hat \pp_{j\aal}$'s are practically identical for $\aal$ large enough. This is because the data do not provide information for all $\pp_{j\aal}$. For the latent classes to which the data provide little information, their posterior means are essentially the same as the prior mean.
	In this procedure, we select the largest $M$ such that
	$$ \hat \pi_{\aal} \ll n^{-1/2} \quad \aal = M+1, M+2, \ldots.$$
    That is, we treat latent classes of very small proportion $(o(\frac{1}{\sqrt{n}}))$  as practically non existing.
	In applying the $K$-means algorithm, we truncate the finite sample to $\hat \pp_{j1}$, $\hat \pp_{j2}$, ..., $\hat \pp_{jM}$.
	Then, the estimated partial information is given by $\aal_1 \overset j = \aal_2$ if $\hat \pp_{j\aal_1}$ and $\hat \pp_{j\aal_2}$ are in the same cluster according to the $K$-means algorithm.
In the following sections, we apply this method to both simulated and real data to assess its performance.

\section{Simulation study}\label{SecSim}

In this section, simulation studies are conducted to assess the performance of
the proposed method and identifiability results. We consider three models for the data generation: NIDA model, reduced NC-RUM model, and LCDM model. The results
are presented assuming all the model parameters are unknown including the
Q-matrix, attribute distribution, and the item response probabilities.

\subsection{NIDA model simulation}

Here, we consider a three-attribute NIDA model. The response probabilities are given by
$$p_{j,\aal} = \prod_{k=1}^K [(1-s_k)^{\alpha_k} g_k ^{1-\alpha_k}]^{q_{jk}}.$$
The $Q$-matrix and parameters $s_k, g_k$ are listed in Table \ref{NIDA}.
There are three attributes and eight classes in total. We consider sample size 2000. The latent class of attribute profile $(1,0,0)$ has mixture probability $\pi_1$,  $(0,1,0)$ has probability $\pi_2$,  $(0,0,1)$ has probability $\pi_3$,  $(1,1,0)$ has mixture probability $\pi_4$,  $(1,0,1)$ has mixture probability $\pi_5$,  $(0,1,1)$ has mixture probability $\pi_6$, $(1,1,1)$ has probability $\pi_7$ and $(0,0,0)$ has mixture probability $\pi_8$.  And we set $\pi_1=\pi_2=\pi_3=\pi_7=0.15, \pi_4=\pi_5=\pi_6=\pi_8=0.1$. This model setup satisfies the identifiability condition.

\begin{table}[tbp]
\centering
\begin{tabular}{|c|ccc|cccccc|}
\hline
\multirow{2}{*}{Item}&\multicolumn{3}{|c|}{\multirow{2}{*}{Q-matrix}} & \multicolumn{6}{|c|}{G and S parameters}\\
&\multicolumn{3}{|c|}{} & s1 & g1 & s2 & g2 & s3 & g3 \\
\hline
 1&1 & 0 & 0 &0.1      & 0.1  & -  &  -  & -  & -  \\
 2&1 & 0 & 0 &0.1      & 0.1  & -  &  -  & -  & -  \\
 3&1 & 0 & 0 &0.1      & 0.1  & -  &  -  & -  & -  \\
 4&0 & 1 & 0 & -      &  -  &  0.1  &  0.2  & -  & -  \\
 5&0 & 1 & 0 & -      & -  &  0.1  &  0.2  & -  & -  \\
 6&0 & 1 & 0 & -      & -  &  0.1  &  0.2  & -  & -  \\
 7&0 & 0 & 1 & -      & -  & -  &  -  & 0.1  &  0.3  \\
 8&0 & 0 & 1 & -      &  -  & -  &  -  & 0.1  & 0.3  \\
 9&0 & 0 & 1 & -      &  -  & -  & -  & 0.1  & 0.3  \\
 10&1 & 1 & 0 &0.1      & 0.5  &  0.1  &  0.5  & -  & -  \\
 11&1 & 0 & 1 &0.1      & 0.5  & -  &  -  &  0.1  & 0.5  \\
 12&0 & 1 & 1 & -      & -  &  0.1  &  0.5  & 0.1  & 0.5  \\
 13&1 & 1 & 1 &0.1      & 0.5  & 0.1  &  0.5  & 0.1  & 0.5  \\
 \hline
\end{tabular}
\caption{Q-matrix and parameters setting for NIDA model simulation}\label{NIDA}
\end{table}

	We fit the model in Section \ref{SecModel} and the parameter estimates (averaging over 100 independent replications) are given in Tables \ref{NIDApb} and \ref{NIDAe}. We can see our estimated response probabilities  are very close to the true.
	The mean squared errors of the estimators are computed based on $100$ independent replicates.
	The mixture probabilities of other classes except these eight are less than 5e-3. Our proposed method can recover the number of classes very well.
    The frequentist properties of estimation of response probabilities are listed in Table \ref{NIDApb}.

	We apply the K-means method to the estimated response probabilities of each item. We use the average Silhouette method to choose the number of clusters.
	We obtain the estimated partial information structure and reconstruct the $Q$-matrix accordingly.
	The proportion of the correctly estimated item's partial information structure is 88.4\%.

\begin{table}[tbp]
\centering
\begin{tabular}{|c|cccccccc|}
\hline
     & C1 & C2 & C3 & C4 & C5 & C6 & C7 & C8 \\
\hline
 Averaged estimates& 0.141  &  0.144  & 0.145  &  0.091  & 0.095  & 0.100 & 0.150 & 0.097  \\
 MSE & 3e-4 & 3e-4  &  4e-4 & 2e-4 & 3e-4 & 2e-4 & 2e-4 & 4e-4  \\
 \hline
\end{tabular}
\caption{The mean and mse of the estimated class probabilities based on 100 simulation runs.}\label{NIDApi}
\end{table}

\begin{table}[tbp]
\centering
\begin{tabular}{|c|cccccccc|cccccccc|}
\hline
\multirow{2}{*}{I}& \multicolumn{8}{|c|}{Estimated prob} & \multicolumn{8}{|c|}{MSE of prob}\\
 & C1 & C2 & C3 & C4 & C5 & C6 & C7 & C8 & C1 & C2 & C3 & C4 & C5 & C6 & C7 & C8\\
\hline
 1& 0.91 & 0.10 & 0.10 & 0.90 & 0.90 & 0.10 &0.90 & 0.10 &  7e-4   & 8e-4  & 6e-4  &  1e-3  & 1e-3  & 8e-4 & 4e-4 & 1e-3\\
 2& 0.91 & 0.10 & 0.09 & 0.90 & 0.90 & 0.10 &0.90 & 0.10&  6e-4   & 6e-4  & 4e-4  &  1e-3  & 9e-4  & 8e-4 & 3e-4& 9e-4\\
 3& 0.90 & 0.10 & 0.10 & 0.90 & 0.90 & 0.11 &0.90 & 0.10&  6e-4   & 5e-4  & 5e-4  &  1e-3  & 1e-3  & 7e-4 & 4e-4& 1e-3\\
 4& 0.20 & 0.90 & 0.20 & 0.90 & 0.20 & 0.90 &0.90 & 0.20&  1e-3   & 7e-4  & 9e-4  &  1e-3  & 1e-3  & 1e-3 & 5e-4& 2e-3\\
 5& 0.20 & 0.90 & 0.20 & 0.90 & 0.20 & 0.89 &0.90 & 0.20&  1e-3   & 9e-4  & 9e-4  &  9e-4  & 2e-3  & 1e-3 & 5e-4& 2e-3\\
 6& 0.20 & 0.90 & 0.20 & 0.90 & 0.20 & 0.90 &0.90 & 0.21&  1e-3   & 7e-4  & 1e-3  &  9e-4  & 1e-3  & 7e-4 & 5e-4& 1e-3\\
 7& 0.30 & 0.30 & 0.90 & 0.29 & 0.90 & 0.89 &0.90 & 0.29&  2e-3   & 2e-3  & 9e-4  &  3e-3  & 1e-3  & 1e-3 & 8e-4& 3e-3\\
 8& 0.30 & 0.30 & 0.89 & 0.30 & 0.90 & 0.89 &0.90 & 0.30&  2e-3   & 2e-3  & 1e-3  &  2e-3  & 1e-3  & 1e-3 & 7e-4& 3e-3\\
 9& 0.30 & 0.30 & 0.90 & 0.30 & 0.89 & 0.90 &0.90 & 0.29&  2e-3   & 2e-3  & 1e-3  &  2e-3  & 1e-3  & 1e-3 & 6e-4& 3e-3\\
 10& 0.45 & 0.45 & 0.25 & 0.81 & 0.45 & 0.44 & 0.81 & 0.25 & 1e-3   & 1e-3  & 8e-4  &  1e-3  & 2e-3  & 2e-3 & 7e-4& 2e-3\\
 11& 0.45 & 0.25 & 0.45 & 0.45 & 0.80 & 0.45 & 0.81& 0.24 &  1e-3   & 1e-3  & 1e-3  &  3e-3  & 1e-3  & 2e-3 & 6e-4& 2e-3\\
 12& 0.24 & 0.45 & 0.45 & 0.45 & 0.45 & 0.81 & 0.81& 0.25 &  1e-3   & 1e-3  & 1e-3  &  2e-3  & 2e-3  & 2e-3 & 9e-4& 2e-3\\
 13& 0.22 & 0.22 & 0.22 & 0.41 & 0.40 & 0.41 & 0.72& 0.12 &  1e-3   & 1e-3  & 1e-3  &  2e-3  & 2e-3  & 2e-3 & 1e-3& 1e-3\\
 \hline
\end{tabular}
\caption{The mean and mse of estimated response probabilities for each latent classes of NIDA model based on 100 simulation runs.}\label{NIDApb}
\end{table}

\begin{table}[tbp]
\centering
\begin{tabular}{|c|cccccc|cccccc|}
\hline
\multirow{2}{*}{Item}& \multicolumn{6}{|c|}{Estimated parameters} & \multicolumn{6}{|c|}{MSE of parameters}\\
 & s1 & g1 & s2 & g2 & s3 & g3  & s1 & g1 & s2 & g2 & s3 & g3 \\
\hline
 1& 0.100 & 0.095 & - & - & - & - &  2e-4   & 2e-4  & -  &  -  & -  & -  \\
 2& 0.101 & 0.096 & - & - & - & - &  2e-4   & 2e-4  & -  &  -  & -  & -  \\
 3& 0.100 & 0.098 & - & - & - & - &  1e-4   & 2e-4  & -  &  -  & -  & -  \\
 4& - & - & 0.101 & 0.199 & - & - &  -  &  -  &  2e-4  &  2e-4  & -  & -  \\
 5& - & - & 0.101 & 0.200 & - & - &   -  & -  &  2e  &  3e-4  & -  & -  \\
 6& - & - & 0.100 & 0.199 & - & - &  -  & -  &  2e-4  &  3e-4  & -  & -  \\
 7& - & - & - & - & 0.103 & 0.299 &  -  & -  & -  &  -  & 3e-4  &  6e-4  \\
 8& - & - & - & - & 0.105 & 0.299 &  -  &  -  & -  &  -  & 2e-4  & 7e-4  \\
 9& - & - & - & - & 0.100 & 0.299 &  -  &  -  & -  & -  & 3e-4  & 6e-4  \\
 10& 0.100 & 0.498 & 0.100 & 0.498  & - & - &  1e-4  & 6e-4  &  1e-4  &  6e-4  & -  & -  \\
 11& 0.101 & 0.497 & - & -  & 0.101 & 0.497 &  1e-4  & 5e-4  & -  &  -  &  1e-4  & 5e-4  \\
 12& - & - & 0.104 & 0.499  & 0.104 & 0.499 & -  & -  &  1e-4  &  9e-4  & 1e-4 & 9e-4  \\
 13& 0.102 & 0.489 & 0.102 & 0.489  & 0.102 & 0.489  & 2e-4  & 2e-3 & 2e-4 & 2e-3 & 2e-4 & 2e-3 \\
 \hline
\end{tabular}
\caption{The mean and mse of the estimated parameters for NIDA model based on 100 simulation runs.}\label{NIDAe}
\end{table}

\subsection{Reduced NC-RUM model simualtion}

In our second simulation setting, we consider a three-attribute NIDA model. The   response probability is
$$p_{j,\aal} = \phi_j \prod_{k=1}^K(r_{jk})^{q_{jk}(1-\alpha_k)}.$$
	The $Q$-matrix and parameters $\phi_j, r_{jk}$ are listed in Table \ref{NCRUM}. We consider sample size 2000. The latent class of   attribute profile $(1,0,0)$ has probability $\pi_1$, $(0,1,0)$ has probability $\pi_2$,  $(0,0,1)$ has probability $\pi_3$,  $(1,1,0)$ has probability $\pi_4$ and  $(1,1,1)$ has probability $\pi_5$.
	We set the probabilities for classes $(1,0,1)$, $(0,1,1)$, and $(0,0,0)$ to zero. Hence, there are five classes in total. $\pi_1=\pi_2=\pi_3=\pi_4 = 1/6, \pi_5=1/3$. This model setup satisfies the identifiability condition.

\begin{table}[tbp]
\centering
\begin{tabular}{|c|ccc|cccc|}
\hline
\multirow{2}{*}{Item}&\multicolumn{3}{|c|}{\multirow{2}{*}{Q-matrix}} & \multicolumn{4}{|c|}{G and S parameters}\\
&\multicolumn{3}{|c|}{} & $\phi$ & r1 & r2 & r3 \\
\hline
 1&1 & 0 & 0 &0.9      & 0.2 & -  &  -   \\
 2&1 & 0 & 0 &0.9      & 0.2  & -  &  -  \\
 3&1 & 0 & 0 &0.9      & 0.2  & -  &  -   \\
 4&0 & 1 & 0 & 0.9      & -  &  0.3  &    \\
 5&0 & 1 & 0 & 0.9      & -  &  0.3  &    \\
 6&0 & 1 & 0 & 0.9      & - &  0.3  &    \\
 7&0 & 0 & 1 & 0.9     & -  & -  &  0.4    \\
 8&0 & 0 & 1 & 0.9      & - & -  &  0.4    \\
 9&0 & 0 & 1 & 0.9      & - & -  & 0.4    \\
 10&1 & 1 & 0 &0.9      & 0.5  &  0.7  & -    \\
 11&1 & 1 & 0 &0.9      & 0.5  & 0.7  &  -    \\
 12&1 & 0 & 1 &0.9      & 0.6  &  -  &  0.4   \\
 13&1 & 0 & 1 &0.9      & 0.6  &  -  &  0.4   \\
 14&0 & 1 & 1 &0.9      & -  &  0.5  & 0.5   \\
 15&0 & 1 & 1 &0.9      & -  &  0.5 &  0.5  \\
 \hline
\end{tabular}
\caption{Q-matrix and parameters setting for reduced NC-RUM model simulation}\label{NCRUM}
\end{table}

The estimated mixture probabilities (averaging over 100 independent replications) and their mean squared errors are listed in Table \ref{NCRUMpb}.
The estimated response probabilities for each class and their mean squared errors are listed in Table \ref{RUMpb}.
We see that the estimated probabilities of the remaining classes other than these 5 classes are very small (below 4e-3).
This suggests that the model identifies exactly 5 classes.

\begin{table}[tbp]
\centering
\begin{tabular}{|c|ccccc|ccccc|}
\hline
\multirow{2}{*}{Item}& \multicolumn{5}{|c|}{Estimated prob} & \multicolumn{5}{|c|}{MSE of prob}\\
 & C1 & C2 & C3 & C4 & C5 & C1 & C2 & C3 & C4 & C5 \\
\hline
 1& 0.91 & 0.17 & 0.17 & 0.90 & 0.90 &   4e-4   & 7e-4  & 5e-4  &  6e-4  & 6e-4  \\
 2& 0.91 & 0.17 & 0.17 & 0.90 & 0.90 &   4e-4   & 8e-4  & 7e-4  &  5e-4  & 2e-3 \\
 3& 0.90 & 0.17 & 0.18 & 0.90 & 0.90 &   4e-4   & 6e-4  & 7e-4  &  5e-4  & 4e-4  \\
 4& 0.26 & 0.90 & 0.27 & 0.90 & 0.90 &   1e-3   & 4e-4  & 6e-4  &  6e-4  & 8e-3  \\
 5& 0.27 & 0.90 & 0.27 & 0.90 & 0.90 &   1e-3   & 3e-4  & 9e-4  &  6e-4  & 3e-4  \\
 6& 0.26 & 0.90 & 0.27 & 0.90 & 0.90 &   1e-3   & 4e-4  & 9e-4  &  6e-4  & 8e-4  \\
 7& 0.36 & 0.36 & 0.90 & 0.36 & 0.90 &   1e-3   & 7e-4  & 3e-4  &  1e-3  & 2e-3  \\
 8& 0.36 & 0.36 & 0.90 & 0.36 & 0.90 &   8e-4   & 9e-4  & 5e-4  &  1e-3  & 4e-3  \\
 9& 0.35 & 0.36 & 0.90 & 0.36 & 0.90 &   1e-3   & 7e-4  & 3e-4  &  1e-3  & 4e-4  \\
 10& 0.63 & 0.45 & 0.31 & 0.90 & 0.90 &   8e-4   & 1e-3  & 7e-4  &  5e-4  & 4e-4  \\
 11& 0.63 & 0.45 & 0.31 & 0.90 & 0.90 &   1e-3   & 1e-3  & 8e-4  &  5e-4  & 1e-3  \\
 12& 0.36 & 0.21 & 0.54 & 0.36 & 0.90 &   7e-4   & 6e-4  & 8e-4  &  1e-3  & 2e-3  \\
 13& 0.35 & 0.21 & 0.54 & 0.36 & 0.90 &   1e-3   & 6e-4  & 9e-4  &  9e-4  & 8e-4  \\
 14& 0.22 & 0.45 & 0.45 & 0.45 & 0.90 &   7e-4   & 8e-4  & 8e-4  &  1e-4  & 8e-3  \\
 15& 0.22 & 0.44 & 0.45 & 0.45 & 0.90 &   7e-4   & 1e-3  & 7e-4  &  1e-3  & 6e-3  \\
\hline
\end{tabular}
\caption{The mean and mse of estimated response probabilities for each latent classes of NC-RUM model based on 100 simulation runs.}\label{RUMpb}
\end{table}

\begin{table}[htbp]\centering
\begin{tabular}{|c|ccccc|}
\hline
    & $C1$   & $C2$  & $C3$  & $C4$  &$C5$  \\
\hline
 mean &0.147    & 0.153  & 0.153  & 0.163  &0.336 \\
 mse  &7e-4    & 4e-4  & 4e-4  & 2e-4  & 3e-4 \\
\hline
\end{tabular}
\caption{The mean and mse of the estimated class probabilities based on 100 simulation runs.}\label{NCRUMpb}
\end{table}

	We also use the Silhouette method to cluster item responses probabilities for each item $j$. The proportion of the correctly estimated item's partial information structure is 88.0\%.
    We also calculate the frequentist properties of parameters of reduced NU-RUM model.
	The result is listed in Table \ref{NCRUMe}.

\begin{table}[tbp]
\centering
\begin{tabular}{|c|cccc|cccc|}
\hline
\multirow{2}{*}{Item}& \multicolumn{4}{|c|}{G and S estimates} & \multicolumn{4}{|c|}{mse of G and S}\\
& $\phi$ & r1 & r2 & r3 & $\phi$ & r1 & r2 & r3 \\
\hline
 1& 0.900 & 0.189 & - & - &1e-4     & 0.001 & -  &  -   \\
 2& 0.902 & 0.194 & - & - &2e-4      & 0.001  & -  &  -  \\
 3& 0.904 & 0.192 & - & - &2e-4      & 0.001  & -  &  -   \\
 4& 0.900 & - & 0.300 & - & 2e-4      & -  &  0.001  &    \\
 5& 0.900 & - & 0.291 & - & 2e-4      & -  &  0.001  &    \\
 6& 0.902 & - & 0.294 & - & 3e-4      & - &  0.001  &    \\
 7& 0.901 & - & -  & 0.395 & 3e-4     & -  & -  &  0.001    \\
 8& 0.901 & - & - & 0.398 & 3e-4      & - & -  &  0.001    \\
 9& 0.901 & - & - & 0.396 & 3e-4      & - & -  & 0.001    \\
 10& 0.899 & 0.491 & 0.710 & - & 2e-4      & 0.002  &  0.003  & -    \\
 11& 0.896 & 0.500 & 0.701 & - & 2e-4     & 0.002  & 0.002  &  -    \\
 12& 0.896 & 0.602 & - & 0.398 & 3e-4      & 0.002  &  -  &  0.002   \\
 13& 0.897 & 0.597 & - & 0.400 & 3e-4      & 0.002  &  -  &  0.002   \\
 14& 0.899 & - & 0.500 &0.494 & 3e-4      & -  &  0.002  & 0.002   \\
 15& 0.896 & - & 0.498 &0.497 & 3e-4      & -  &  0.002 &  0.002  \\
 \hline
\end{tabular}
\caption{The mean and mse of the estimated parameters for reduced NC-RUM model based on 100 simulations.}\label{NCRUMe}
\end{table}

\subsection{LCDM model simulation}

In our third simulation setting, we consider a three-attribute LCDM model which is defined as

\begin{equation}\label{lcdm}\nonumber
p_{j,\aal} = \frac{\exp(\boldsymbol{\lambda}_j^{T} h(\aal,\qq_j) -\eta)}{1+\exp(\boldsymbol{\lambda}_j^{T} h(\aal,\qq_j) -\eta_j)},
\end{equation}
where
$$ \boldsymbol{\lambda}_j^{T} h(\aal,\qq_j) = \sum\limits_{k=1}^K \lambda_{jk}(\alpha_k q_{jk}) + \sum\limits_{k=1}^K \sum\limits_{\tilde{k}>k} \lambda_{jk\tilde{k}}(\alpha_k \alpha_{\tilde{k}}q_{jk}q_{j\tilde{k}}) + \cdots,$$
and $\boldsymbol{\lambda}_j, \eta_j$ are prespecified parameters.

The $Q$-matrix and parameters are listed in Table \ref{LCDM}.
	We generate date sets containing 2000 observations. The latent class of attribute profile $(1,0,0)$ has mixture probability $\pi_1$, $(0,1,0)$ has probability $\pi_2$,  $(0,0,1)$ has probability $\pi_3$, $(1,1,0)$ has probability $\pi_4$,  $(1,0,1)$ has probability $\pi_5$, $(0,1,1)$ has probability $\pi_6$,  $(1,1,1)$ has probability $\pi_7$, and  $(0,0,0)$ has probability $\pi_8$.
	There are eight classes in total. $\pi_1=\pi_2=\pi_3=\pi_4=\pi_5=\pi_6=\pi_7=\pi_8=1/8$. This model setup satisfies the identifiability condition.

\begin{table}[tbp]
\centering
\begin{tabular}{|c|ccc|c|}
\hline
\multirow{2}{*}{Item}&\multicolumn{3}{|c|}{\multirow{2}{*}{Q-matrix}} & \multicolumn{1}{|c|}{ Parameters}\\
&\multicolumn{3}{|c|}{} & $\boldsymbol{\lambda}, \eta$ \\
\hline
 1&1 & 0 & 0  & $\lambda_1 =4, \eta=-2$         \\
 2&1 & 0 & 0  & $\lambda_1 =4, \eta=-2$         \\
 3&1 & 0 & 0  & $\lambda_1 =4, \eta=-2$          \\
 \hline
 4&0 & 1 & 0  & $\lambda_2 =4, \eta=-2$      \\
 5&0 & 1 & 0  & $\lambda_2 =4, \eta=-2$       \\
 6&0 & 1 & 0  & $\lambda_2 =4, \eta=-2$          \\
 \hline
 7&0 & 0 & 1  & $\lambda_3 =4, \eta=-2$          \\
 8&0 & 0 & 1  & $\lambda_3 =4, \eta=-2$       \\
 9&0 & 0 & 1  & $\lambda_3 =4, \eta=-2$       \\
 \hline
 10&1 & 1 & 0  & $\lambda_1 =2, \lambda_2=2, \lambda_{1,2}=0, \eta=-2$       \\
 11&1 & 1 & 0  & $\lambda_1 =2, \lambda_2=2, \lambda_{1,2}=0, \eta=-2$      \\
 12&1 & 0 & 1  & $\lambda_1 =2, \lambda_3=2, \lambda_{1,3}=0, \eta=-2$      \\
 13&1 & 0 & 1  & $\lambda_1 =2, \lambda_3=2, \lambda_{1,3}=0, \eta=-2$      \\
 14&0 & 1 & 1  & $\lambda_2 =2, \lambda_3=2, \lambda_{2,3}=0, \eta=-2$      \\
 15&0 & 1 & 1  & $\lambda_2 =2, \lambda_3=2, \lambda_{2,3}=0, \eta=-2$       \\
 \hline
 \multirow{2}{*}{16}&\multirow{2}{*}{1} & \multirow{2}{*}{1} & \multirow{2}{*}{1} & $\lambda_1 =1, \lambda_2=1, \lambda_3=1, \eta=-2 $\\
 & & & & $\lambda_{1,2}=\lambda_{1,3}=\lambda_{2,3}=0,\lambda_{1,2,3}=1$ \\
 \hline
\end{tabular}
\caption{Q-matrix and parameter setting for LCDM model simulation.}\label{LCDM}
\end{table}

The estimated mixture probabilities (averaging over 100 independent replications) and their mean squared errors are given in Table \ref{LCDMpi}.
The estimated probabilities of the remaining classes other than these 8 classes are all less than 3e-3.
Also, the estimated response probabilities for each item are listed in Table \ref{LCDMpb}
\begin{table}[htbp]\centering
\begin{tabular}{|c|cccccccc|}
\hline
    & $C1$   & $C2$  & $C3$  & $C4$  &$C5$ & $C6$ & $C7$ & $C8$   \\
\hline
 mean &0.123    & 0.121  & 0.122  & 0.119  &0.121 & 0.121 & 0.125 & 0.125\\
 mse &1.1e-4    & 1.1e-4  & 1.7e-4  & 1.1e-4  & 1.0e-4 & 1.3e-4 & 1.1e-4 & 1.0e-4\\
\hline
\end{tabular}
\caption{The mean and mse of the estimated class probability of each class based on 100 simulation runs.}\label{LCDMpi}
\end{table}

\begin{table}[tbp]
\centering
\begin{tabular}{|c|cccccccc|cccccccc|}
\hline
\multirow{2}{*}{I}& \multicolumn{8}{|c|}{Estimated prob} & \multicolumn{8}{|c|}{MSE of prob}\\
 & C1 & C2 & C3 & C4 & C5 & C6 & C7 & C8 & C1 & C2 & C3 & C4 & C5 & C6 & C7 & C8\\
\hline
 1& 0.88 & 0.12 & 0.12 & 0.88 & 0.88 & 0.12 & 0.88 & 0.12&  6e-4   & 7e-4  & 5e-4  &  6e-4  & 5e-4  & 4e-4 & 1e-2 & 7e-4\\
 2& 0.88 & 0.12 & 0.12 & 0.88 & 0.88 & 0.12 & 0.87 & 0.13&  6e-4   & 7e-4  & 6e-4  &  6e-4  & 6e-4  & 7e-4 & 3e-3 & 9e-4 \\
 3& 0.88 & 0.12 & 0.12 & 0.88 & 0.88 & 0.12 & 0.88 & 0.12&  6e-4   & 5e-4  & 6e-4  &  5e-4  & 5e-4  & 6e-4 & 7e-4 & 7e-4 \\
 4& 0.12 & 0.88 & 0.12 & 0.88 & 0.12 & 0.88 & 0.88 & 0.12&  7e-4   & 6e-4  & 6e-4  &  6e-4  & 8e-4  & 6e-4 & 3e-3 & 9e-4 \\
 5& 0.12 & 0.88 & 0.12 & 0.88 & 0.12 & 0.88 & 0.88 & 0.12&  6e-4   & 7e-4  & 5e-4  &  6e-4  & 8e-4  & 7e-4 & 1e-3 & 9e-4 \\
 6& 0.12 & 0.88 & 0.12 & 0.88 & 0.12 & 0.88 & 0.87 & 0.13&  5e-4   & 6e-4  & 6e-4  &  5e-4  & 6e-4  & 6e-4 & 2e-3 & 6e-4 \\
 7& 0.12 & 0.12 & 0.88 & 0.12 & 0.88 & 0.88 & 0.88 & 0.12&  6e-4   & 5e-4  & 7e-4  &  7e-4  & 7e-4  & 6e-4 & 3e-3 & 3e-3 \\
 8& 0.12 & 0.12 & 0.88 & 0.12 & 0.88 & 0.88 & 0.88 & 0.13&  8e-4   & 4e-4  & 6e-4  &  6e-4  & 5e-4  & 6e-4 & 8e-4 & 4e-3 \\
 9& 0.12 & 0.12 & 0.88 & 0.12 & 0.88 & 0.88 & 0.87 & 0.12&  7e-4   & 5e-4  & 6e-4  &  8e-4  & 7e-4  & 6e-4 & 1e-2 & 1e-2 \\
 10& 0.50 & 0.51 & 0.12 & 0.88 & 0.50 & 0.49 & 0.88 & 0.12&  8e-4   & 1e-3  & 6e-4  &  5e-4  & 1e-3  & 1e-3 & 8e-4 & 1e-3 \\
 11& 0.50 & 0.50 & 0.12 & 0.88 & 0.50 & 0.50 & 0.87 & 0.12&  1e-3   & 1e-3  & 6e-4  &  7e-4  & 1e-3  & 1e-3 & 5e-3 & 4e-3 \\
 12& 0.50 & 0.12 & 0.50 & 0.49 & 0.88 & 0.50 & 0.88 & 0.12&  1e-3   & 6e-4  & 1e-3  &  1e-3  & 6e-4  & 1e-3 & 2e-3 & 8e-3 \\
 13& 0.50 & 0.12 & 0.50 & 0.50 & 0.88 & 0.49 & 0.88 & 0.12&  1e-3   & 4e-4  & 9e-4  &  1e-3  & 5e-4  & 1e-3 & 2e-3 & 4e-3 \\
 14& 0.12 & 0.50 & 0.50 & 0.49 & 0.50 & 0.88 & 0.88 & 0.12&  6e-4   & 1e-3  & 2e-3  &  1e-3  & 1e-3  & 6e-4 & 3e-3 & 3e-3 \\
 15& 0.12 & 0.50 & 0.50 & 0.50 & 0.50 & 0.88 & 0.88 & 0.12&  6e-4   & 1e-3  & 1e-3  &  1e-3  & 1e-3  & 6e-4 & 9e-4 & 1e-3 \\
 16& 0.27 & 0.27 & 0.27 & 0.50 & 0.50 & 0.50 & 0.87 & 0.12&  1e-3   & 1e-3  & 9e-4  &  1e-3  & 1e-3  & 1e-3 & 1e-2 & 9e-4 \\
 \hline
\end{tabular}
\caption{The mean and mse of estimated response probabilities for each latent classes of LCDM model based on 100 simulation runs.}\label{LCDMpb}
\end{table}

We use the Silhouette method to cluster item response probabilities for each item $j$ to estimate the partial information structure. The  proportion of correctly estimated item partial information structure is 96.5\%.
We calculate the mean and mean square error of the estimators of the LCDM model.
The results are listed in Table \ref{LCDMe}.

\begin{table}[tbp]
\centering
\begin{tabular}{|c|c|c|}
\hline
\multirow{2}{*}{Item}&\multicolumn{1}{|c|}{Estimates} & \multicolumn{1}{|c|}{ MSE of Estimates}\\
& $\boldsymbol{\lambda}, \eta$ & $mse_{\lambda}, mse_{\eta}$ \\
\hline
 1& $\lambda_1 = 4.00, \eta=-1.99$ & $m_{\lambda_1}=0.03 , m_{\eta}=0.01 $         \\
 2& $\lambda_1 =3.97, \eta=-1.97$ & $m_{\lambda_1}=0.03 , m_{\eta}=0.01 $          \\
 3& $\lambda_1 =3.99, \eta=-1.99$ & $m_{\lambda_1}=0.03 , m_{\eta}=0.02 $           \\
 \hline
 4& $\lambda_2 =3.97, \eta=-1.98$ & $m_{\lambda_2}=0.03 , m_{\eta}=0.02 $       \\
 5& $\lambda_2 =3.99, \eta=-1.99$ & $m_{\lambda_2}=0.03 , m_{\eta}=0.02 $        \\
 6& $\lambda_2 =3.98, \eta=-1.99$ & $m_{\lambda_2}=0.04 , m_{\eta}=0.02 $           \\
 \hline
 7& $\lambda_3 =4.00, \eta=-2.01$ & $m_{\lambda_3}=0.04 , m_{\eta}=0.02 $           \\
 8& $\lambda_3 =3.98, \eta=-1.99$ & $m_{\lambda_3}=0.03 , m_{\eta}=0.01 $        \\
 9& $\lambda_3 =3.96, \eta=-1.99$ & $m_{\lambda_3}=0.03 , m_{\eta}=0.01 $        \\
 \hline
 10& $\lambda_1 =2.00, \lambda_2=2.00, \lambda_{1,2}=-0.03, \eta=-2.00$&$m_{\lambda_1}=.05, m_{\lambda_2}=.05, m_{\lambda_{1,2}}=.11, m_{\eta}=.03$   \\
 11& $\lambda_1 =2.02, \lambda_2=2.00, \lambda_{1,2}=-0.01, \eta=-2.01$&$m_{\lambda_1}=.03, m_{\lambda_2}=.04, m_{\lambda_{1,2}}=.08, m_{\eta}=.02$   \\
 12& $\lambda_1 =1.98, \lambda_3=1.96, \lambda_{1,3}=0.01, \eta=-1.97$&$m_{\lambda_1}=.03, m_{\lambda_3}=.04, m_{\lambda_{1,3}}=.08, m_{\eta}=.02$   \\
 13& $\lambda_1 =1.97, \lambda_3=1.96, \lambda_{1,3}=0.02, \eta=-1.96$&$m_{\lambda_1}=.04, m_{\lambda_3}=.04, m_{\lambda_{1,3}}=.10, m_{\eta}=.03$   \\
 14& $\lambda_2 =2.00, \lambda_3=1.99, \lambda_{2,3}=0.03, \eta=-2.00$&$m_{\lambda_2}=.05, m_{\lambda_3}=.05, m_{\lambda_{2,3}}=.09, m_{\eta}=.03$   \\
 15& $\lambda_2 =1.96, \lambda_3=1.96, \lambda_{2,3}=0.03, \eta=-1.97$&$m_{\lambda_2}=.05, m_{\lambda_3}=.05, m_{\lambda_{2,3}}=.09, m_{\eta}=.03$    \\
 \hline
 \multirow{2}{*}{16}& $\lambda_1 =0.97, \lambda_2=0.95,\lambda_3=0.96,\eta= -1.98$&$m_{\lambda_1}=.05,m_{\lambda_2}=.07,m_{\lambda_3}=.08,m_{\eta}=.07$\\
 & $\lambda_{1,2}=0.06 \lambda_{1,3}=0.06\lambda_{2,3}= 0.08 ,\lambda_{1,2,3}=0.86 $&$m_{\lambda_{1,2}}=.15, m_{\lambda_{1,3}}=.16, m_{\lambda_{2,3}}=.11 ,m_{\lambda_{1,2,3}}=.31 $ \\
 \hline
\end{tabular}
\caption{The mean and mse of the estimated parameters for LCDM model based on 100 simulations.}\label{LCDMe}
\end{table}

\section{Real data analysis}\label{SecReal}

We apply the proposed method to a subset of the National Epidemiological Survey
on Alcohol and Related Conditions (NESARC) (Grant et al., 2003 \cite{Nesarc}) concerning social phobia. There are in total 13 questions that are presented in Table \ref{t1}.
We fit the latent Dirichlet allocation model and estimate the partial information structure via the procedure in Section 4. The results are summarized as follows.

To obtain meaningful and stable estimates, we consider large latent classes whose probabilities are over $2\%$. According to the fitted model, there are five such latent classes, each of which is over $2\%$ of the population. The estimated posterior probability of each class is $\pi_1=0.11$, $\pi_2=0.37$, $\pi_3=0.10$, $\pi_4=0.20$, and $\pi_5=0.15$. They add up to 93\% of the population.
The estimated probabilities are presented in Table \ref{posterior}.

We apply the K-means method to the item response probabilities of each item. The number of clusters are selected by the silhouette method.
The partial information is then obtained via this cluster analysis. The results are summarized in Table \ref{cluster}. From Table \ref{cluster}, we can see the 13 items may be divided into three groups according to their functioning. Items 1-4 can differentiate between classes 1,2,3,5 and class 4. Items 5-8 differentiate classes 1,5 and classes 2,3,4. Items 9-13 differentiate classes 3,5 and classes 1,2,4.
Furthermore, we can see that items 2,3,7,8,12,13 differentiate multiple groups, indicating that these items may be more informative.

We further construct a parametrization of the item response function based on the estimated partial information. It turns out to be the Reduced NC-RUM model with number of attributes $K=3$ whose item response function is defined as
\begin{equation}\label{NC-RUM}
    p_{j,\aal}= \phi_j \prod\limits_{k=1}^{K}(r_{jk})^{q_{jk(1-\alpha_k)}},
\end{equation}
where $\phi_j$ is the correct response probability for subjects who possess all required attributes and $r_{jk}$, $0 < r_{jk} < 1$, is the penalty parameter for not possessing the $k$-th attribute. The corresponding item parameters are $\mathbf{\theta} = \{\phi_j, r_{jk}: j=1,\cdots,J, k=1,\cdots,K\}$.
Based on Table \ref{cluster}, we could parameterize each class, $C1 = (1,1,0)$, $C2 = (1,0,0)$, $C3=(1,0,1)$, $C4=(0,0,0)$, and $C5 = (1,1,1)$.
Also, an estimated $Q$-matrix under this parameterization is provided in Table \ref{Qmat}. The estimated parameters for the reduced NC-RUM model under estimated $Q$-matrix are given in Table \ref{parameterize}.
In this parameterization, we obtain an approximate interpretation that attribute 1 corresponds to ``public performance", attribute 2 to ``scrutiny", and  attribute 3 to ``interaction".


Table \ref{posterior} presents posterior mean of the response probability for each item under each class.  Based on the loadings in this table, we may interpret these latent classes as follows.
Class 1 is strongly associated with items 1-4 which are related to public performance and items 5-8 which are related to close scrutiny. Thus, this group is characterized by strong fear of public performance and close scrutiny.
Class 2 has high response probabilities of items 1-4. This shows that it possesses attribute 1 only, i.e. those people in this class are afraid of public performance, but not close scrutiny and interaction.
Class 3 has relatively high response probabilities of items 1-4 and 9-13 and relatively low response probability of items 5-8, which means that this class is more likely related with attributes 1 and 3. In other words, people from group 3 may have ``fears'' of public performance and interaction with other people.
Class 4 is loosely associated with all items. The all item response probabilities are low which indicates that this class may not be connected with any attributes.
Class 5 has the highest response probabilities among all items. This shows that class 5 is associated with all items. Hence, people from class 5 are likely to possess 3 attributes. In other words, class 5 corresponds to, using a technical term, the generalized social anxiety disorder subtype (``fears most social situations''; American Psychiatric Association, 1994 \cite{APA}).

\begin{table}[htbp]\centering
\begin{tabular}{ll}
\hline
ID          & Have you ever had a strong fear or avoidance of  \\
\hline
1          & speaking in front of other people?     \\
2          & taking part or speaking in class?    \\
3          & taking part or speaking at a meeting?      \\
4          & performing in front of other people?          \\
5          & being interviewed?        \\
6          & writing when someone watches?          \\
7          & taking an important exam?     \\
8          & speaking to an authority Lgure?    \\
9          & eating or drinking in front of other people?    \\
10         & having conversations with people you don't know well?    \\
11         & going to parties and social gatherings?    \\
12         & dating?   \\
13         & being in a small group situation?   \\
\hline
\end{tabular}
\caption{The content of 13 items for the social anxiety disorder data.}\label{t1}
\end{table}

\begin{table}[htbp]\centering
\begin{tabular}{|l|ccccc|}
\hline
   item   &          C1&    C2&    C3&    C4&   C5 \\
\hline
    1 &        0.93 & 0.93 & 0.95 & 0.25 & 0.97 \\
    2 &        0.95 & 0.78 & 0.7 & 0.07 & 0.97 \\
    3 &        0.85 & 0.55 & 0.59 & 0.02 & 0.96 \\
    4 &        0.86 & 0.56 & 0.82 & 0.11 & 0.94 \\
    5 &        0.73 & 0.15 & 0.34 & 0.12 & 0.71 \\
    6 &        0.25 & 0.06 & 0.07 & 0.05 & 0.4 \\
    7 &        0.65 & 0.18 & 0.34 & 0.12 & 0.74 \\
    8 &       0.65 & 0.18 & 0.37 & 0.15 & 0.78 \\
    9 &        0.03 & 0.01 & 0.37 & 0.04 & 0.39 \\
    10 &        0.23 & 0.17 & 0.88 & 0.12 & 0.93 \\
    11 &       0.25 & 0.18 & 0.87 & 0.14 & 0.91 \\
    12 &        0.33 & 0.18 & 0.33 & 0.13 & 0.60 \\
    13 &        0.10 & 0.01 & 0.11 & 0.01 & 0.43 \\
\hline
\end{tabular}
\caption{The estimated probability matrix based on latent Dirichlet allocation model for the
social anxiety disorder data. Each row corresponding to the item response probability for each class.}\label{posterior}
\end{table}

\begin{table}[htbp]\centering
\begin{tabular}{|l|ccccc|}
\hline
  item   &          C1&    C2&    C3&    C4&   C5 \\
\hline
   1&         $\bullet$ & $\bullet$ & $\bullet$ & $\circ$ & $\bullet$ \\
   2&         $\dag$ & $\bullet$ & $\bullet$ & $\circ$ & $\dag$ \\
   3&         $\dag$ & $\bullet$ & $\bullet$ & $\circ$ & $\dag$ \\
   4&         $\dag$ & $\bullet$ & $\dag$ & $\circ$ & $\ddag$ \\
   5&         $\bullet$ & $\circ$ & $\circ$ & $\circ$ & $\bullet$ \\
   6&         $\bullet$ & $\circ$ & $\circ$ & $\circ$ & $\bullet$ \\
   7&         $\dag$ & $\circ$ & $\bullet$ & $\circ$ & $\ddag$ \\
   8&         $\dag$ & $\circ$ & $\bullet$ & $\circ$ & $\ddag$ \\
   9&         $\circ$ & $\circ$ & $\bullet$ & $\circ$ & $\bullet$ \\
   10&         $\circ$ & $\circ$ & $\bullet$ & $\circ$ & $\bullet$ \\
   11&         $\circ$ & $\circ$ & $\bullet$ & $\circ$ & $\bullet$ \\
   12&         $\bullet$ & $\circ$ & $\bullet$ & $\circ$ & $\dag$ \\
   13&         $\bullet$ & $\circ$ & $\bullet$ & $\circ$ & $\dag$ \\
\hline
\end{tabular}
\caption{The estimated cluster matrix based on estimated posterior probability matrix by using k-means method. Here we use symbol $\circ, \bullet, \dag, \ddag$ to represent the level 1 to 4 respectively. Level 1 represents the lowest probability level, the level 4 represent the highest probability level.}
\label{cluster}
\end{table}

\begin{table}[htbp]\centering
\begin{tabular}{l}
$ Q = \left(
          \begin{array}{ccc}
            1 & 0 & 0 \\
            1 & 1 & 0 \\
            1 & 1 & 0 \\
            1 & 1 & 1 \\
            0 & 1 & 0 \\
            0 & 1 & 0 \\
            0 & 1 & 1 \\
            0 & 1 & 1 \\
            0 & 0 & 1 \\
            0 & 0 & 1 \\
            0 & 0 & 1 \\
            0 & 1 & 1 \\
            0 & 1 & 1 \\
          \end{array}
        \right)$ \\
\end{tabular}
\caption{The estimated Q-matrix based on the three dimensional Reduced NC-RUM model for the
social anxiety disorder data.}\label{Qmat}
\end{table}

\begin{table}[htbp]\centering
\begin{tabular}{|c|llll|}
\hline
Item & $\phi_{\cdot}$          & $r_{\cdot,1}$  & $r_{\cdot,2}$  & $r_{\cdot,3}$\\
\hline
1 & $\phi_1=0.95$          & $r_{1,1}=0.26$ & -  &  -   \\
2 & $\phi_2=0.95$          & $r_{2,1}=0.10$ & $r_{2,2}=0.79$ & - \\
3 & $\phi_3=0.9$           & $r_{3,1}=0.03 $ & $r_{3,2}= 0.67 $ &  -   \\
4 & $\phi_4=0.95$          & $r_{4,1}=0.16$& $r_{4,2}=0.80$ & $ r_{4,3}=0.80 $        \\
5 & $\phi_5=0.7$           & - & $r_{5,2}=0.21$ & -      \\
6 & $\phi_6=0.3$           & - &$r_{6,2}=0.20 $ & -        \\
7 & $\phi_7=0.75$          & - &$r_{7,2}=0.40$ &$r_{7,3}=0.80$    \\
8 & $\phi_8=0.8$           & - &$r_{8,2}=0.37$ &$r_{8,3}=0.75$   \\
9 & $\phi_9=0.4$           & - & - &$r_{9,3}=0.08$   \\
10 & $\phi_{10}=0.9$        & - & - &$r_{10,3}=0.22$   \\
11 & $\phi_{11}=0.9$        & - & - &$r_{11,3}=0.17 $  \\
12 & $\phi_{12}=0.6$        & - &$r_{12,2}=0.55 $ &$r_{13,3}=0.55 $ \\
13 & $\phi_{13}=0.45$       & -&$r_{13,2}=0.20 $ &$r_{13,3}=0.20 $ \\
\hline
\end{tabular}
\caption{The table of parameters for the three dimensional Reduced NC-RUM model for the
social anxiety disorder data. Some parameters in model could not be estimated due to structure of Reduced NC-RUM model.}\label{parameterize}
\end{table}

\section{Discussion}

This paper deals with certain fundamental issues in latent class models under a general framework. In particular, it does not require the usual $Q$-matrix structure which is commonly assumed in most diagnostic classification models.
We established four theorems on identifiability under various model settings including binary response case, multi-categorical response case and multi-level attribute case. In particular, we provide easy-to-check sufficient conditions in Theorem 4 that are applicable to a general class of latent class models. Further in Theorem 5, we show the existence of a consistent estimator which can asymptotically identify partial information structure of items. We construct an appropriate estimator by using latent Dirichlet allocation method, which does not require pre-specification of the number of latent classes. The simulation results show the proposed method works well under a variety of settings.

  There are some recent works on identifiability for certain latent class models. Xu, Shang, Zhang (2016) \cite{Xu1,Xu2,Xu3} give several identifiability results for $Q$-matrix and parameters under the restricted latent class setting with binary item responses. Their main result provides a slightly weaker sufficient condition version of our corollary 1. This is because that our results are applicable to general latent class models with multi-categorical responses and diagnostic classification models with multi-level attributes.

 Because the existing methods require specification of a particular $Q$-matrix based model, it is of interest to develop model checking methods for any departure from the model specifications.
 Identifiability and estimation methods in this paper are applicable to the more general settings and, therefore, could be useful for developing such model checking methods.

\bibliographystyle{plain}
\bibliography{IRT2,bibstat,CDM,general}

\newpage
\appendix
\section{Technical proofs}

\begin{lemma}[Kruskal \cite{Kruskal}]\label{Lem}

  Suppose $A,B,C,\bar{A},\bar{B},\bar{C}$ are six matrices with $R$ columns.
  There exist integers $I_0$, $J_0$, and $K_0$ such that $I_0+J_0+K_0 \geq 2R+2$.
  In addition, every $I_0$ columns of $A$ are linearly independent, every $J_0$ columns of $B$ are linearly independent, and every $K_0$ columns of $C$ are linearly independent.
  Define a triple product to be a three-way array $[A,B,C] = (d_{ijk})$ where $d_{ijk}=\sum_{r=1}^{R} a_{ir} b_{ir} c_{kr}$.
  Suppose that the following two triple products are equal $[A,B,C]=[\bar{A},\bar{B},\bar{C}]$. Then, there exists a column permutation matrix $P$,  we have $\bar{A}=AP\Lambda, \bar{B}=BPM, \bar{C}=CPN$, where $\Lambda, M, N$ are diagonal matrices such that $\Lambda MN =$ identity.
  Column permutation matrix is a matrix acts on the righthand side of another matrix and permutes the columns of that matrix.

\end{lemma}

\begin{proof}[Proof of Theorem \ref{ThmBinaryBinary}]
For each item $j$,
let
$p_{j\aal} = P(Y_j = 1 | \aal).$
that takes two possible values. Let $p_{j-}$ or $p_{j+}$ be these two values.
 According to Lemma \ref{Lem}, it is sufficient to show that the $T$-matrices corresponding to the three subsets of items $T_{I_1}$, $T_{I_2}$, and $T_{I_3}$ are all of full column rank.

Suppose that there are $n_i$ items in $I_i$. For each item $j\in I_i$, define
$$\mathbf{P}_{j}=\left(
\begin{array}{cc}
p_{j-} & p_{j+} \\
1-p_{j-} & 1-p_{j+} \\
\end{array}\right).
$$
We further define
$$\mathcal{P}_i=\bigotimes_{j\in I_i} \mathbf{P}_j$$
which is a $2^{n_i}$ by $2^{n_i}$ matrix.
Because $p_{j-} \neq p_{j+}$, each $\mathbf{P}_{j}$ is a full-rank matrix and is of rank 2.
Thus, $\mathcal{P}_c$ is rank $2^{n_i}$ matrix and is a full-rank matrix.
Each column of $T_{I_i}$ is precisely one of the column vector in $\mathcal{P}_i$. In addition, there is no identical columns in $T_{I_i}$, thus its columns vectors  are linearly independent. Thus, $T_{I_i}$ is of full column rank.

We construct three groups of items $\tilde{I}_1=I_1, \tilde{I}_2=I_2$ and $\tilde{I}_3=\{1,\ldots,J \}\backslash (I_1 \bigcup I_2)$. These three groups are non-overlapping and $I_3 \subset \tilde{I}_3$. Notice that $T_{\tilde{I}_1}=T_{I_1}, T_{\tilde{I}_2}=T_{I_2}$ and $T_{I_3}$ is a submatrix of $T_{\tilde{I}_3}$. Therefore, $T_{\tilde{I}_1}, T_{\tilde{I}_2}$ and $T_{\tilde{I}_3}$ are all full column rank. We define
$$W = [T_{\tilde{I}_1}\Lambda, T_{\tilde{I}_2}, T_{\tilde{I}_3}],$$
where $\Lambda$ is a $M$ by $M$ diagonal matrix with $\pi_{\aal}$ being its $\aal$-th element. It is not hard to see that every entry of array $W$ corresponds to a probability $P(Y^1=y^1,\ldots, Y^J=y^J)$.

 Suppose that there is another decomposition of $W$ say $W=[T_{\tilde{I}_1}^{'}\Lambda^{'}, T_{\tilde{I}_2}^{'}, T_{\tilde{I}_3}^{'}]$. Notice that each $T_{\tilde{I}_i}$ has rank $M$ and $M+M+M \geq 2M+2$ provided $M \geq 2$. Then we apply Lemma \ref{Lem} and obtain that $T_{\tilde{I}_1}\Lambda=T_{\tilde{I}_1}^{'} \Lambda^{'} P A$, $T_{\tilde{I}_2}=T_{\tilde{I}_2}^{'} P B$, and $T_{\tilde{I}_3}=T_{\tilde{I}_3}^{'} P C$. Here, $A, B$, and $C$ are all diagonal matrix with $ABC = I$ and $I$ is an identity and $P$ is a column permutation matrix. Each column of $T_{\tilde{I}_i}$ and $T_{\tilde{I}_i}^{'}$ corresponds to a probability distribution and thus sums up to one. It means $A, B$ and $C$ must be identity matrix. Hence, we conclude that $T_{\tilde{I}_1} \Lambda=T_{\tilde{I}_1}^{'}\Lambda^{'}P$ which implies $\Lambda = P^{'}\Lambda^{'} P$. Then, we have $T_{\tilde{I}_1}=T_{\tilde{I}_1}^{'}P$, $T_{\tilde{I}_2}=T_{\tilde{I}_2}^{'}P$ and $T_{\tilde{I}_3}=T_{\tilde{I}_3}^{'}P$. This is equivalent that the item parameters $p_{j\aal}^k$ and the latent class population $\pi_{\aal}$ are identifiable up to a permutation of the class label.
\end{proof}

\bigskip

\begin{proof}[Proof of Corollary \ref{ThmB}]
Without loss of generality, we assume that the first, second, and third $K$ rows of $Q$ each form an identity matrix.
The attributes are binary and each of the first $3K$ items only depends on one attribute. Thus, their item response function $p_{j\aal}$ can only take two possible values.
Furthermore, we  divide these $3K$ items into 3 groups $I_1=\{1,\dots,K\}$, $I_2 = \{K+1,\cdots,2K\}$, and $I_3= \{2K+1,\cdots,3K\}$. It is straightforward to check that these three subsets of items satisfy condition A1 in Theorem \ref{ThmBinaryBinary}.
The corollary is an application of Theorem \ref{ThmBinaryBinary}.
\end{proof}

\bigskip

\begin{proof}[Proof of Theorem \ref{ThmBinary}]
Under condition $B1$, we define
\begin{equation}\label{pf2eq1}\nonumber
    \mathbf{P}_j = \left(
            \begin{array}{cc}
              p_{j-}^1 & p_{j+}^1 \\
              p_{j-}^2 & p_{j+}^2 \\
              \vdots & \vdots \\
              p_{j-}^{k_j} & p_{j+}^{k_j} \\
            \end{array}
          \right),
\end{equation}
whose column vectors are the two positive $P_{j\aal}$.
For each $I_i$, we define
$$\mathcal{P}_i=\bigotimes_{j\in I_i} \mathbf{P}_j$$
which is a $\prod\limits_{j\in I_i}k_j$ by $2^{n_i}$ matrix. $n_i$ is the number of items in $I_i$.
Each $\mathbf{P}_{j}$ is a full column rank matrix of rank 2.
Thus, $\mathcal{P}_i$ is rank of $2^{n_i}$ matrix and is a full column rank matrix.

 Each column vector of $T_{I_i}$ is a column vector of $\mathcal{P}_i$. We can show that for two classes $\aal_1 \neq \aal_2$, $\aal_1$-th and $\aal_2$-th columns of $T_{I_i}$ are not identical.
 We prove this by contradiction. Suppose that they are the same.
 It is easy to see that the $\aal_l$-th column in $T_{I_i}$ has the form
$$\bigotimes\limits_{j \in I_i}
 \left(
     \begin{array}{c}
       p_{j\aal_l}^1 \\
       p_{j\aal_l}^2 \\
       \vdots \\
       p_{j\aal_l}^{k_j} \\
     \end{array}
   \right),$$
$l=1,2$. So
\begin{equation}\label{pf2eq2}
 \bigotimes\limits_{j \in I_i}
 \left(
     \begin{array}{c}
       p_{j\aal_1}^1 \\
       p_{j\aal_1}^2 \\
       \vdots \\
       p_{j\aal_1}^{k_j} \\
     \end{array}
   \right)
   =
   \bigotimes\limits_{j \in I_i}
 \left(
     \begin{array}{c}
       p_{j\aal_2}^1 \\
       p_{j\aal_2}^2 \\
       \vdots \\
       p_{j\aal_2}^{k_j} \\
     \end{array}
   \right).
\end{equation}

However, we can find item $j^{\ast} \in I_i$ such that
$$p_{j^{\ast}\aal_1}^1 + ... + p_{j^{\ast}\aal_1}^k \neq p_{j\aal_2}^1 + ... + p_{j^{\ast}\aal_2}^k$$
for all $k=1,...,k_{j^{\ast}}-1$ which means
$$ \left(
     \begin{array}{c}
       p_{j^{\ast}\aal_1}^1  \\
       p_{j^{\ast}\aal_1}^2  \\
       \vdots  \\
       p_{j^{\ast}\aal_1}^{k_{j^{\ast}}}  \\
     \end{array}
   \right)
   \neq
   \left(
     \begin{array}{c}
       p_{j^{\ast}\aal_2}^1 \\
       p_{j^{\ast}\aal_2}^2 \\
       \vdots \\
       p_{j^{\ast}\aal_2}^{k_{j^{\ast}}} \\
     \end{array}
   \right)
.$$
It contradicts with equation (\ref{pf2eq2}) due to the fact that two different marginal distributions of item $j^{\ast}$ leads to the two different joint distributions. Hence, each column of $T_{I_i}$ is precisely one column of $\mathcal{P}_i$. $T_{I_i}$ is of full column rank with rank $M$ as a result.
Then $M+M+M \geq 2M+2$ whenever $M \geq 2$. We apply Lemma \ref{Lem} and use the same argument as in the proof of Theorem \ref{ThmBinaryBinary}.
\end{proof}
\bigskip

\begin{proof}[Proof of Theorem \ref{ThmGeneral}]
There exist three non-overlapp subsets of items $I_1,I_2$, and $I_3$ such that $I_1\bigcup I_2 \bigcup I_3=\{1,\ldots,J\}$. We write the three-way array $W=[T_{I_1}\Lambda,T_{I_2},T_{I_3}]$, where $T_{I_1},T_{I_2}$, and $T_{I_3}$ are the $T$-matrices of subsets $I_1, I_2$, and $I_3$ respectively and $\Lambda$ is a $\prod_{j=1}^{J}k_j$ by $\prod_{j=1}^{J}k_j$ diagonal matrix with $\aal$-th diagonal element being $\pi_{\aal}$. Thus $W$ is a $\prod_{j \in I_1} k_j$ by $\prod_{j \in I_2} k_j$ by $\prod_{j \in I_3} k_j$ array. It is not hard to see that $W(\yy_1,\yy_2,\yy_3) = \sum\limits_{\aal}t_{\yy_1\aal}^1 t_{\yy_2\aal}^2 t_{\yy_3\aal}^3$, where $t_{\yy_i \aal}^i$ is the $(\yy_i, \aal)$-th element of matrix $T_{I_i}$. In other words, $W(\yy_1,\yy_2,\yy_3)=P(\YY=(\yy_1,\yy_2,\yy_3))$.

Suppose that there exists another set of parameters of the model giving the same distribution; that is, another decomposition of $W=[T_{I_1}^{'} \Lambda^{'},T_{I_2}^{'},T_{I_3}^{'}]$. Because $T_{I_i}$ are all full column rank. By applying the Lemma \ref{Lem}, we have that $T_{I_1}\Lambda=T_{I_1}^{'} \Lambda^{'} P A$, $T_{I_2}=T_{I_2}^{'} P B$, and $T_{I_3}=T_{I_3}^{'} P C$. Here, $A, B$, and $C$ are all diagonal matrix with $ABC = I$ and $I$ is an identity and $P$ is a column permutation matrix.

The sum of each column of $T_{I_2},T_{I_3},T_{I_2}^{'},$ and $T_{I_3}^{'}$ equals 1. Then, $B$ and $C$ must be both identity matrices. As a result, $A$ is identity too. Due to the same reason that the sum of each column of $T_{I_1}$ and $T_{I_1}^{'}$ is 1, we have $\Lambda = P^{'} \Lambda^{'} P$ and $T_{I_1} = T_{I_1}^{'} P$. Besides, $T_{I_2}=T_{I_2}^{'} P$ and $T_{I_3}=T_{I_3}^{'} P$. We conclude that all parameters are identifiable up to a permutation of the columns.
\end{proof}

\bigskip

\begin{proof}[Proof of Theorem \ref{ThmGeneralDCM}]
According to Theorem \ref{ThmGeneral}, it is sufficient to find three non-overlap subsets of items $I_1, I_2$, and $I_3$ such that $I_1 \bigcup I_2 \bigcup I_3 = \{1,\ldots,J\}$ and their corresponding $T$-matrices $T_{I_1}, T_{I_2},T_{I_3}$ are all full column rank.

We construct the three subsets as follows: $I_1 = \bigcup_{k=1}^{K}I_{1,k}, I_2 = \bigcup_{k=1}^{K}I_{2,k}, I_3=J \backslash (I_1 \bigcup I_2)$.
Then we need to show that $I_1, I_2, I_3$ are non-overlap and their $T$-matrix $T_{I_1}, T_{I_2}, T_{I_3}$ are of full column rank.

 We know that $I_{i,k} \bigcap I_{j,l}=\emptyset$ for all $i\neq j$, $k\neq l, k,l \in \{1,\ldots,K\}$ and therefore $(\bigcup_{k=1}^{K}I_{i,k}) \bigcap$ $(\bigcup_{k=1}^{K}I_{j,k}) = \emptyset$. This also implies that $\bigcup_{k=1}^{K}I_{3,k} \subset I_3$ and $I_1 \bigcap I_2 =\emptyset = I_3 \bigcap I_1 = I_3 \bigcap I_2 = \emptyset$. Hence, $I_1, I_2$ and $I_3$ are non-overlap.

 Next, we need to prove $T_{I_i}, i=1,2,3$ are of full column rank. Notice that $\bigcup_{k=1}^{K}I_{i,k} \subset I_i$. Thus the rank of $T_{\bigcup_{k=1}^{K}I_{i,k}}$ is less than or equals the rank of $T_{I_3}$. Thus if we can prove $T_{\bigcup_{k=1}^{K}I_{i,k}}$ is of full column rank then $T_{I_3}$ is also of full column rank. As a result, we only need to show $T_{\bigcup_{k=1}^{K}I_{i,k}}$ are of full column rank.

 Recall that the class label $\aal=(\alpha^1,\ldots,\alpha^K)$ and $\alpha^k \in \{1,\ldots,d_k\}$. It is straightforward to see that $T_{\bigcup_{k=1}^{K}I_{i,k}}=\bigotimes_{k=1}^K T_{i,k}$ since each column of $\bigotimes_{k=1}^K T_{i,k}$ is indexed by $(\alpha^1,\ldots,\alpha^K)$ and each row in $\bigotimes_{k=1}^K T_{i,k}$ is indexed by all the possible values $(y^1,\ldots,y^K)$.  By the property of tensor product, the rank of $\bigotimes_{k=1}^K T_{i,k}$ equals the product of the rank of $T_{i,k}$. That is $ \hbox{rank}(\bigotimes_{k=1}^K T_{i,k}) = \prod_{k=1}^K d_k$. The number of columns in $\bigotimes_{k=1}^K T_{i,k}$ is also $\prod_{k=1}^K d_k$. Thus $T_{\bigcup_{k=1}^{K}I_{i,k}}$ is of full column rank.
\end{proof}

\bigskip

\begin{proof}[Proof of Theorem \ref{ThmPI}]
	It is sufficient to construct a consistent estimator of the partial information. Notice that the estimator does not have to be practically implementable.
The strategy is to first consider the maximum likelihood estimator and merge the estimated item response probabilities based on their asymptotic properties.

	Recall that $p_{j\aal}^y = P(Y^j = y |\aal)$  is the response probability to item $j$ for latent class $\aal$. Let
\begin{equation}\label{MLE}
L(\yy; \pp, \pi) =  \sum_{\aal \in \mA}\Big \{ \prod_{j=1}^J p_{j\aal}^{y^j} \pi_{\aal}\Big \}.
\end{equation}
be the likelihood of a single observation, where
$$\pp = (p_{j\aal}^{y^j}: 1\leq j \leq J, \aal \in \mathcal M, y^j \in \{1,...,k_j\})$$
and $\pi =(\pi _{\aal}: \aal \in \mathcal M)$.
Then, the maximum likelihood estimator is defined as
$$(\hat \pp,\hat \pi) =\arg\max_{\pp,\pi} \prod_{i=1}^n L(\yy_i; \pp, \pi).$$
According to the identifiability results in theorems and  the asymptotic property of the $M$-estimator (Chapter 5.1 of \cite{Van}), $(\hat \pp,\hat \pi) $ converges weakly to the true parameter. Furthermore, according to chapter 5.3 \cite{Van}, the MLE is asymptotically normally distributed. Thus, for each $j$, $\aal$, and $y$, we have
\begin{equation}\label{l2}
\hat p_{j\aal}^{y} - p_{j\aal}^{y} = O_p(n^{-1/2}).
\end{equation}
We say a random sequence $ a_n = O_p(n^{-1/2})$ if $\sqrt n a_n$ is tight.
Notice that the identifiability is subject to a permutation of the latent class labels.
To simplify notation, we assume that the class labels of $\hat{p}_{j\aal}^y$ have been arranged in an appropriate order. Otherwise, we need to write $\hat{p}_{j\aal}^y - p_{j\aal}^y = O_p(n^{-1/2})$. Thus, we proceed assuming that the permutation $\lambda$ is identity.

We now construct an estimator of the partial information for each item.
The basic idea is that if $[\aal_1]_j = [\aal_2]_j$, then $p_{j\aal_1}^{y}= p_{j\aal_2}^{y}$ for all $y$.
Together with \eqref{l2}, we have that
\begin{equation}\label{merge}
d_j (\aal_1,\aal_2) = \sum_{y=1}^{k_j} (\hat p_{j\lambda(\aal_1)}^{y} - \hat p_{j\lambda(\aal_2)}^{y})^2 = O_p(n^{-1}).
\end{equation}
Based on this fact, we define an equivalent class such that
\begin{equation}\label{eq}
\aal_1 \overset j = \aal_2 \quad \textrm{if} \quad d_j(\aal_1,\aal_2) \leq n^{-1/2}.
\end{equation}
Based on \eqref{merge}, we have that
$$P(\aal_1 \overset j = \aal_2) \to 1$$
as $n \to \infty$.
If $[\aal_1]_j \neq [\aal_2]_j$, then there exists an $\varepsilon$ and $y$ such that $(\hat p_{j\lambda(\aal_1)}^{y} - \hat p_{j\lambda(\aal_2)}^{y})^2 > \varepsilon$ and thus
$$P(\aal_1 \overset j \neq \aal_2) \to 1$$
as $n \to \infty$.
Let ``$\langle~~\rangle_j$'' be the canonical map of the estimated equivalence class as in \eqref{eq}.
Based on the above argument, we have that for each $j$,
$$P(\langle \aal \rangle_j = [\aal]_j) \to 1$$
as the sample size $n\to \infty$.
Hence, the estimation of equivalence classes is the same as the true one up to a permutation.
\end{proof}


\appendix
\section{Sliced sampler for latent class model}

We now present the sliced sampler for the simulation from the posterior distribution of model \eqref{like}, \eqref{priorP}, and \eqref{priorPi}.
The likelihood function is
$$\prod _{i=1}^n \Big \{\sum_{\aal =1}^\infty \pi_{\aal} \prod_{j=1}^J \prod_{y=1}^{k_j} (p_{j\aal}^y)^{I(y_{ij} = y)}\Big \}.$$
For each observation, we augment an independent index $u_i$ following the uniform distribution in $[0,1]$.
Thus, the complete data likelihood is
$$L(\pp,\pi; \yy_1,...,\yy_n, \aal_1,...,\aal_n, \uu)= \prod _{i=1}^n \Big \{ I(u_i < \pi_{\aal_i}) \prod_{j=1}^J \prod_{y=1}^{k_j} (p_{j\aal}^y)^{I(y_{ij} = y)}\Big \}.$$
With this augmentation scheme, a Gibbs sampler iterates according the following conditional distributions.

\begin{enumerate}
\item Update $u_i$, for $i=1,...,n$, by sampling from the conditional posterior, $U(0, \pi_{\aal_i})$.

\item For $h=1,...,M$ where $M = \max\{\aal_1,...,\aal_n\}$, update $\pp_{j\aal}$ from the full conditional posterior distribution
$$\textrm{Dirichlet}\Big(1+ \sum_{i:\aal_i = \aal} I(y_{ij}=1), ..., 1+ \sum_{i:\aal_i = \aal} I(y_{ij}=k_j)\Big)$$

\item For $h=1,...,M$, update $V_{\aal}$ from the conditional distribution that is Beta$(1,\beta)$ truncated to the interval
$$\left[\max_{i:\aal_i = \aal} \Big\{\frac{u_i}{\prod_{l<\aal} (1-V_l)}\Big\},
1-\max_{i:\aal_i >\aal} \Big\{\frac{u_i}{V_{\aal}\prod_{l<\aal,l \neq \aal} (1-V_l)}\Big\}\right].$$

\item Update each $\aal_i$ from the multinomial conditional distribution
$$P(\aal_i = \aal | ...) = \frac{I(\aal \in A_i) \prod_{j=1}^J p_{j\aal}^{y_{ij}}}{\sum_{l\in A_i} \prod_{j=1}^J p_{j\aal}^{y_{ij}}}$$
where $A_i = \{\aal: \pi_{\aal} > u_i\}$.

\item Assuming a gamma$(1,1)$ hyperprior for $\beta$, update $\beta$ by its conditional posterior
$$\hbox{gamma}(1+M,1-\sum\limits_{\aal=1}^{M}\log(1-V_{\aal})). $$
\end{enumerate}

\end{document}